\documentclass[11pt, a4paper]{amsart}

\makeatletter
\newcommand*{\rom}[1]{\expandafter\@slowromancap\romannumeral #1@}
\makeatother
\usepackage{graphicx,color}
\usepackage{psfrag}
\usepackage{amscd}
\usepackage{amsthm,amsmath,amssymb,graphics,hyperref}
\usepackage{amsmath,amssymb,hyperref}
\usepackage{amsfonts}
\usepackage{enumerate}
\usepackage{amsmath, amsthm}
\usepackage{amscd}
\usepackage[titletoc,title]{appendix}
\usepackage[T1]{fontenc}
\usepackage[utf8]{inputenc}
\input xy
\xyoption{all}

      \theoremstyle{plain}
      \newtheorem{theorem}{Theorem}[section]
      \newtheorem{lemma}[theorem]{Lemma}
      \newtheorem{corollary}[theorem]{Corollary}
      \newtheorem{proposition}[theorem]{Proposition}
      \theoremstyle{definition}
      \newtheorem{definition}[theorem]{Definition}

      \theoremstyle{remark}
      \newtheorem{remark}[theorem]{Remark}

\newcommand{\rsl}{{\mathrm{PGL}}}
\newcommand{\pgl}{\mathrm{PGL}(m, \mathbb R)}
\newcommand{\cal}{\mathcal}
\newcommand{\rp}{\mathbb{RP}}
\newcommand{\ra}{\rightarrow}
\newcommand{\flagv}{\mathrm{Flag}(\sigma_{mod})}
\newcommand{\R}{\mathbb{R}}

\let\cal\mathcal

\def\kim{\color{black}}

      \makeatletter
      \def\@setcopyright{}
      \def\serieslogo@{}
      \makeatother

\begin{document}

%



 \author{Inkang Kim}
   \address{School of Mathematics, Korea Institute for Advanced Study, Hoegiro 85, Dongdaemun-gu,
   Seoul, 02455, Republic of Korea}
   \email{inkang@kias.re.kr}

   \author{Sungwoon Kim}
   \address{Department of Mathematics, Jeju National University, 102 Jejudaehak-ro, Jeju, 63243, Republic of Korea}
   \email{sungwoon@jejunu.ac.kr}






   \title[]{Primitive stable representations in higher rank semisimple Lie groups}


\begin{abstract}
We study primitive stable representations of free groups into higher rank semisimple Lie groups and their properties. Let $\Sigma$ be a compact, connected, orientable surface (possibly with boundary) of negative Euler characteristic. We first verify $\sigma_{mod}$-regularity for convex projective structures and positive representations. Then we show that the holonomies of convex projective structures and positive representations on $\Sigma$ are all primitive stable if $\Sigma$ has one boundary component.
\end{abstract}



\footnotetext[3]{I. Kim gratefully acknowledges the partial support of grant (NRF-2017R1A2A2A05001002) and KIAS Individual Grant (MG031408), and a warm support of UC Berkeley during his stay. S. Kim gratefully acknowledges supports from the 2020 scientific promotion program by Jeju National University and the Basic Science Research Program through the National Research Foundation of Korea funded by the Ministry of Education, Science and Technology (NRF-2015R1D1A1A09058742).}



   \thanks{}
   \thanks{}

   \dedicatory{}

   \date{}


   \maketitle



\section{Introduction}

Recently, much attention has been paid to the generalization of convex cocompact groups in rank one symmetric spaces to
higher rank symmetric spaces. The successful story along this line is Anosov representations which was introduced by Labourie \cite{La} and developed further by Guichard and Wienhard \cite{GW}.
In the recent paper \cite{KLP2} of Kapovich, Leeb and Porti,  more geometric criteria for Anosov representations are given. Among those properties, the concept of Morse actions of word hyperbolic groups is outstanding \cite{KLP2}.
On the other hand, Minsky \cite{Minsky} proposed the notion of primitive stable representations in real hyperbolic 3-space.
Combining these notions, one can extend the notion of primitive stable representations of free groups to higher rank semisimple Lie groups \cite{GGKW, KLP2}.
 In the case of $\mathrm{PSL}(2,\mathbb C)$, see \cite{K,KL} for criteria of primitive stability for handlebodies and its generalization to compression bodies.

Let $G$ be a higher rank semisimple Lie group without compact factors, $X$ the associated symmetric space, and $\Gamma$ a free group of rank $r$.
The definition of primitive stable representation has been already mentioned by Gu\'eritaud-Guichard-Kassel-Wienhard in \cite[Remark 1.6(b)]{GGKW}.
In the paper, in order to define primitive stable representation, we will use the concept of Morse quasigeodesic which is introduced by Kapovich, Leeb and Porti in \cite{KLP2}.
A representation $\rho:\Gamma \rightarrow G$ is said to be \emph{primitive stable} if any bi-infinite geodesic in the Cayley graph of $\Gamma$ defined by
a primitive element is mapped under the orbit map to a uniformly Morse quasigeodesic in $X$.
Let $\mathcal{PS}(\Gamma,G)$ be the set of conjugacy classes of primitive stable representations. 
Then, the openness of $\mathcal{PS}(\Gamma,G)$ directly follows from the stability of the Morse property in \cite[Section 7]{KLP}. Furthermore it easily follows from the argument of Minsky \cite{Minsky} that the action of the outer automorphism $Out(\Gamma)$ of $\Gamma$ on $\mathcal{PS}(\Gamma,G)$ is properly discontinuous.
Hence one easily has the following.

\begin{theorem}\label{thm:di} Let $\Gamma$ be a free group and $G$  a semisimple Lie group without compact factors. Then the set $\mathcal{PS}(\Gamma,G)$ of primitive stable representations of $\Gamma$ in $G$ is open in the character variety of $\Gamma$ in $G$, and the action of the outer automorphism $Out(\Gamma)$ of $\Gamma$ on $\mathcal{PS}(\Gamma,G)$ is properly discontinuous.
\end{theorem}

{\kim Theorem \ref{thm:di} means that the set of primitive stable representations of $\Gamma$ in $G$ is a domain of discontinuity for the action of $Out(\Gamma)$ on the character variety of $\Gamma$ in $G$ which is strictly larger than the set of Anosov representations of $\Gamma$. Indeed, there have been many studies on domain of discontinuity for the $Out(\Gamma)$-action.
In higher rank,} Canary-Lee-Stover \cite{CLS} studied amalgam Anosov representations for a one-ended torsion free hyperbolic group $\Gamma$ to show that they form a domain of discontinuity for the action of $Out(\Gamma)$. {\kim However the class of amalgam Anosov representations does not include the primitive stable representations of free groups.}

From the definition of primitive stable representation, it is clear that any Anosov representation is primitive stable.
The main point of the paper is to give  concrete examples of primitive stable representations which are not Anosov representations.
Here are our main theorems.

\begin{theorem} \label{mainthm}Let $\Sigma$ be compact, connected, orientable surface with one boundary component and negative Euler characteristic. Then the holonomy representations of convex projective structures on $\Sigma$ are primitive stable. Furthermore, every positive representation of $\Sigma$ in $\mathrm{PGL}(m,\R)$ is primitive stable.
\end{theorem}



There are three kinds of the holonomies of convex projective structures; Anosov representations, Positive representations and Quasi-hyperbolic representations. The holonomies of convex projective structures with hyperbolic geodesic boundary are  Anosov representations as well as positive representations, and the holonomies of convex projective structures of finite Hilbert volume with cusps are positive but not Anosov representations. Every quasi-hyperbolic representation, which is the holonomy of a convex projective structure with quasi-hyperbolic element, admits an equivariant continuous map $\partial_\infty \pi_1(\Sigma) \ra \flagv$. But the map is not antipodal due to the  property of quasi-hyperbolic element. See Section \ref{convex} for more details.
Hence quasi-hyperbolic representations are neither Anosov representations nor positive representations but they are primitive stable. Theorem \ref{mainthm} gives examples of primitive stable representations which are neither Anosov representations nor positive representations. Positive representations with unipotent element are not Anosov representations.
Hence we have the following corollary.

\begin{corollary}
Let $\Gamma$ be a non-abelian free group of even rank. Then there is an open subset of the character variety of $\Gamma$ in $\pgl$, strictly larger than the set of Anosov representations, which is $Out(\Gamma)$-invariant, and on which $Out(\Gamma)$ acts properly discontinuously.
\end{corollary}

What is essential in the proof of Theorem \ref{mainthm} is $\sigma_{mod}$-regularity. We refer the reader to \cite[Section 4.2]{KLP2} for details concerning $\sigma_{mod}$-regularity.
The regularities of Anosov representations are {\kim actually due to Labourie \cite{La} and Guichard-Wienhard \cite{GW}}. However it is not easy to verify the regularity for representations which are not Anosov.
To study the $\sigma_{mod}$-regularity, we first deal with convex projective structures.

\begin{theorem}\label{cr}
Let $\Sigma$ be a compact, connected, orientable surface with boundary and negative Euler characteristic. Let $\rho :\pi_1(\Sigma) \ra \mathrm{PGL}(3,\mathbb R)$ be the holonomy of a convex projective structure on the interior of $\Sigma$. {\kim If $\rho$ is a positive representation, $\rho$ is uniformly $\sigma_{mod}$-regular. Otherwise $\rho$ is $\sigma_{mod}$-regular but not uniformly $\sigma_{mod}$-regular.}
\end{theorem}

{\kim For positive representations $\rho :\pi_1(\Sigma) \ra \mathrm{PGL}(3,\mathbb R)$ with only hyperbolic boundary holonomy, Theorem \ref{cr} can be obtained by doubling the convex projective surface. Indeed, this case is just the Anosov case. For the other convex projective structures, Theorem \ref{cr} is not immediate. We emphasize that we here deal with the $\sigma_{mod}$-regularity for convex projective structures which are not Anosov.}

Furthermore, we generalize a tool in $\mathrm{PGL}(3,\mathbb R)$ to check $\sigma_{mod}$-regularity to general $\pgl$.
More precisely, we find out how to see $\sigma_{mod}$-regularity for a given boundary embedded representation. Here a boundary embedded representation means a representation $\rho :\Gamma \ra G$ for a (relatively) hyperbolic group $\Gamma$ which admits a $\rho$-equivariant homeomorphism $\xi : \partial_\infty \Gamma \ra G/P$ for some parabolic subgroup $P$ of $G$. By applying the tool to positive representations, we have the following.

\begin{proposition}\label{positiveregular}
Every positive representation in $\pgl$ is $\sigma_{mod}$-regular.
\end{proposition}

Both Theorem \ref{cr} and Proposition \ref{positiveregular} might be important in studying the notion of relatively Anosov representations in future since positive representations with unipotent elements have been regarded as important examples in developing the notion of relatively Anosov representation.

{\kim 
Let $S=\mathbb H^2/\Gamma$ be a hyperbolic surface possibly with boundary. We say that a representation $\rho : \Gamma\to \pgl$ is \emph{type-preserving} if $\rho$ sends hyperbolic elements  to positive hyperbolic (i.e. diagonalizable with distinct positive real eigenvalues),  unipotent elements to unipotent.
The $i$th simple root length for a positive hyperbolic element $\gamma\in \Gamma$ is defined by
\[\ell_i(\gamma):=\ln \frac{\lambda_i(\rho(\gamma))}{\lambda_{i+1}(\rho(\gamma))}\]
where $\lambda_1(\rho(\gamma))>\lambda_2(\rho(\gamma))>\cdots >\lambda_m(\rho(\gamma))$ are the eigenvalues of  a positive hyperbolic element $\rho(\gamma)$. The simple $\ell_i$-spectrum is defined as the set of $i$th simple root lengths of oriented closed geodesics on $S$. 

If a positive representation is Anosov, the discreteness of its simple $\ell_i$-spectrum is obtained from the property of Anosov representation. On the other hand, the discreteness is not obvious for the other positive representations which have at least one unipotent boundary holonomy. As a corollary of Proposition \ref{positiveregular}, we show that:

\begin{corollary}\label{cor:1.6} 
Let $S=\mathbb H^2/\Gamma$ be a hyperbolic surface with boundary and $\rho: \Gamma\to \pgl$ be a type-preserving positive representation.
Then the simple $\ell_i$-spectrum for $\rho$ is discrete for all $i=1,\ldots,m-1$.
\end{corollary}

Huang and Sun proved a weaker version of this corollary in \cite[Theorem 1.22]{HS} with a different method. 
They show the discreteness of the simple $\ell_i$-spectrum for oriented simple closed geodesics, not for all oriented closed geodesics.

}

\section{Preliminaries}


In this section, we collect basic notions and results that are necessary to define and study primitive stable representations in higher rank semisimple Lie groups.

\subsection{Free group, primitive element and Whitehead Lemma}\label{blocking}
 Given a generating set $\{s_1,\ldots,s_r\}$ of a non-abelian free
group $\Gamma$, let $S$ denote the set $\{s_1^\pm,\ldots,s_r^\pm\}$. The Cayley graph $\mathcal C(\Gamma,S)$ is a 1-dimensional tree
whose vertices are words in $S$, and two words $v$ and $w$ are
connected by a length one edge if and only if $v=ws$ for some $s \in S$.
A group $\Gamma$ acts on its Cayley graph on the left, i.e., for $\gamma \in \Gamma$,
$$ (\gamma, w)\rightarrow   \gamma w.$$
This action is an isometry since for any $s \in S$,
$$d(w, v)=d(sw, sv),$$  where $d(\cdot,\cdot)$ is the distance induced by the word metric on  $\mathcal C(\Gamma,S)$.

An element of $\Gamma$ is called \emph{primitive} if it is a member of a free generating set. Each conjugacy class $[w]$ in $\Gamma$ determines
$\overline w$, the periodic word determined by concatenating infinitely many copies of a cyclically reduced representative of $w$. 
A cyclically reduced word $w$ defines a unique invariant line $\widetilde w$ through the origin $e$ in the Cayley graph.
Each $\overline w$ lifts to  a $\Gamma$-invariant set of bi-infinite geodesics in $\mathcal C(\Gamma,S)$. Indeed, this invariant set is the orbit of
$\widetilde w$ under the action of $\Gamma$.

Let $\overline{\mathcal P}$ denote the set consisting of $\overline w$ for conjugacy classes $[w]$ of primitive elements of $\Gamma$ and
 ${\mathcal P}$  the set of all bi-infinite geodesics $q :\mathbb Z \rightarrow \Gamma$ in $\mathcal C(\Gamma,S)$ lifted from $\overline w$ for all primitive elements $w$ in $\Gamma$.
Let $\mathcal B$ be the set of  all bi-infinite geodesics in the Cayley graph of $\Gamma$.
 It is known that $\mathcal P$ is a strict subset of ${ \mathcal B}$.

For a word $\gamma \in \Gamma$, the Whitehead graph $Wh(\gamma, S)$ is the graph with $2r$ vertices labeled $s_1,s_1^{-1},\ldots,s_r,s_r^{-1}$, and edge from $v$ to $w^{-1}$ for each string $vw$ that appears in $\gamma$ or in a cyclic permutation of $\gamma$. Whitehead proved that
for a cyclically reduced word $\gamma$, if $Wh(\gamma,S)$ is connected and has no cutpoint, then $\gamma$ is not primitive. We refer the reader to \cite{Wh1,Wh2} for more information. One says that a reduced word $\gamma$ is \emph{primitive-blocking} if it does not appear as a subword of any cyclically reduced primitive word and \emph{blocking} if some power $\gamma^m$ is primitive-blocking. An immediate corollary of the Whitehead lemma is that for a once-punctured surface, the cusp curve $c$ is blocking. Indeed $c^2$ cannot appear as a subword of any primitive element, since $Wh(c^2, S)$ is a cycle \cite{Minsky}.

\subsection{Cartan projection, Lyapunov projection}
Let $\mathfrak a$ be the set of real  traceless diagonal $m\times m$ matrices and $\bar{\mathfrak a}^+$ be the set of elements of $\mathfrak a$ whose diagonal entries are in nonincreasing order. Let $G=\pgl$ and $K=\mathrm{PO}(m)$. Then the Cartan decomposition $K(\exp \bar{\mathfrak a}^+)K$ means that each $g \in G$ can be written by $g=k (\exp \mu(g) ) k'$ for some $k, k' \in K$ and a unique $\mu(g) \in \bar{\mathfrak a}^+$. A Cartan projection $$\mu : G \ra \bar{\mathfrak a}^+$$ is defined as the map sending $g\in G$ to $\mu(g)$.

An element $g \in G$ can be uniquely written as $$g=ehu$$ where $e$ is elliptic (all its complex eigenvalues have modulus 1), $h$ is hyperbolic (all the eigenvalues are real and positive) and $u$ is unipotent, and all three commute \cite[Lemma 7.1]{Hel}. This decomposition is called the \emph{Jordan decomposition} of $g$.
The \emph{Lyapunov projection} $\lambda : G \rightarrow \bar{\mathfrak a}^+$ is induced from the Jordan decomposition: for $g\in G$, $\lambda(g) \in \bar{\mathfrak a}^+$ is a unique element such that $\exp(\lambda(g))$ is conjugate to the hyperbolic component $h$ of the Jordan decomposition $g=ehu$.


\subsection{Morse quasigeodesic}\label{Morse}
Let $G$ be a semisimple Lie group without compact factors and $K$ be a maximal compact subgroup of $G$. Let $X$ be the associated symmetric space of noncompact type.  
Let $\partial_\infty X$ denote the visual boundary of $X$ {\kim i.e. the set of all asymptotic classes of geodesic rays of $X$.}
{\kim The symmetric space $X$ associated to $G$ is a Hadamard manifold with the symmetric Riemannian metric $d_X$ and hence the visual boundary $\partial_\infty X$ is homeomorphic to the sphere of $\dim(X)-1$. There is a natural topology on $X\cup\partial_\infty X$, namely the \emph{cone topology}. 
The cone topology on  $X\cup\partial_\infty X$ is generated by the open sets in $X$ and open cones which are defined by 
\[C_x(\xi,\epsilon)=\{ y\in X\cup \partial_\infty X \ | \ y\neq x \text{ and }\angle_x(\xi, y)<\epsilon \} \]
for $x\in X$ and $\xi\in \partial_\infty X$. For more details, we refer the reader to \cite[Section 3]{BGS}.
}

Let $W$ be the Weyl group acting on a model maximal flat $F_{mod} \cong \mathbb R^\mathrm{rank(G)}$ of $X$ and on the model apartment $a_{mod}=\partial_\infty F_{mod}\cong S^{\mathrm{rank}(G)-1}$ where $\mathrm{rank}(G)$ denotes the real rank of $G$.
The pair $(a_{mod},W)$ is the spherical Coxeter complex associated with $X$.
Then the spherical model Weyl chamber is defined as the quotient $\sigma_{mod}=a_{mod}/W$.  The natural projection $\theta : \partial_\infty X \rightarrow \sigma_{mod}$ restricts to an isometry on every chamber $\sigma \subset \partial_\infty X$.
For a chamber $\sigma \in \partial_\infty X$ and a point $x \in X$, the \emph{Weyl sector} $V(x,\sigma)$ is defined as the union of rays emanating from $x$ and asymptotic to $\sigma$. The Euclidean model Weyl chamber $\Delta$ is defined as the cone over $\sigma_{mod}$ with tip at the origin.

For two points $x,y\in X$, the $\Delta$-valued distance $d_\Delta(x,y)$ is defined as follows: Choose a maximal flat $F$ containing $x$ and $y$. Identifying $F$ isometrically with $F_{mod}$, regard $x$ and $y$ as points in $F_{mod}$. Then $d_{\Delta}(x,y)$ is defined by $$d_\Delta(x,y)=proj(y-x) \in \Delta$$ where $proj : F_{mod} \ra F_{mod}/W \cong \Delta$ is the quotient map. Note that in general, $d_\Delta$ is not symmetric. The resulting $\Delta$-valued distance $d_\Delta(x,y)$ does not depend on the choices of $F$. 
Let $\partial_{Tits}X$ be the Tits boundary of $X$. For a simplex $\tau\subset \partial_{Tits} X$, st($\tau$) is the smallest subcomplex of $\partial_{Tits} X$ containing all chambers $\sigma$ such that $\tau\subset\sigma$. The open star ost($\tau$) is the union of all open simplices whose closures intersect int($\tau$) nontrivially. For a face $\tau_{mod}$ of the model Weyl chamber $\sigma_{mod}$, ost($\tau_{mod}$) denotes its open star in $\sigma_{mod}$.  
A point $\xi\in \partial_{Tits} X$ is called \emph{$\tau_{mod}$-regular} if $\theta(\xi)\in \text{ost}(\tau_{mod})$ where $\theta$ is the type map defined as the canonical projection map $$\theta:\partial_\infty X\rightarrow \partial_\infty X/ G=\sigma_{mod}.$$ If  $\theta(\xi)\in\Theta$ for a compact set $\Theta\subset \text{ost}(\tau_{mod})$, then $\xi\in\partial_\infty X$ is said to be \emph{$\Theta$-regular}. A geodesic segment $xy$ is called \emph{$\Theta$-regular} (resp. \emph{$\tau_{mod}$-regular}) if it is contained in a geodesic ray $x\xi$ with $\xi$ $\Theta$-regular (resp. $\tau_{mod}$-regular). Note that if $\Theta$ is $\iota$-invariant, $xy$ is $\Theta$-regular if and only if $yx$ is $\Theta$-regular.  Here $\iota$ is $-w_0$ where $w_0$ is the longest element of the Weyl group. A $\Theta$-star of a simplex $\tau$ of type $\tau_{mod}$ is $\text{st}_\Theta (\tau)=\text{st}(\tau)\cap \theta^{-1}(\Theta)$. Then the $\Theta$-cone $V(x,\text{st}_\Theta(\tau))$, which is a union of geodesic rays starting at $x$ and asymptotic to $\text{st}_\Theta (\tau)$, is convex in $X$. For more details, see \cite[Section 4.2]{KLP2}.

\begin{definition}[Regular sequence, \cite{KLP2}]\label{def:primitive}
{\kim A sequence $x_n \to \infty$ in $X$ is $\tau_{mod}$-{\sl regular} if for some (hence any) $x\in X$,
$$d(d_\Delta(x, x_n), V(0,\sigma_{mod}-\mathrm{ost}(\tau_{mod})))\rightarrow +\infty \text{ as }n\ra \infty.$$ 
A sequence $g_n\rightarrow \infty$ in $G$ is $\tau_{mod}$-{\sl regular} if some (hence any) orbit $(g_nx)$ is $\tau_{mod}$-regular.}
\end{definition}
In particular, when $\tau_{mod}=\sigma_{mod}$, obviously $V(0,\sigma_{mod}-\mathrm{int}(\sigma_{mod}))=\partial \Delta$.


\begin{definition}[Morse quasigeodesic, \cite{KLP2}] A continuous map $p:I\rightarrow X$ is called an \emph{$(L,A, \Theta, D)$-Morse quasigeodesic} if it is an $(L,A)$-quasigeodesic and for all $t_1,t_2\in I$, the subpath $p|_{[t_1,t_2]}$ is $D$-close to a $\Theta$-diamond $\Diamond_\Theta(x_1,x_2)$ with $d_X(x_i, p(t_i))\leq D$ where $\Theta$ is an $\iota$-invariant $\tau_{mod}$-convex compact set and $\iota(\tau_{mod})=\tau_{mod}$. Here a \emph{$\Theta$-diamond} of a $\Theta$-regular segment $xy$ is
$$\Diamond_\Theta(x,y)=V(x,\text{st}_\Theta(\tau_+))\cap V(y,\text{st}_\Theta(\tau_-))$$ where $\tau_+$ (resp. $\tau_-$) is a unique simplex  of
type $\tau_{mod}$ such that the geodesic ray $x\xi$ (resp. $y\xi$) containing $xy$ with $\xi\in$ st($\tau_+)$ (resp. $\xi\in$ st($\tau_-)$). A continuous map $p:I\rightarrow X$ is called an \emph{$(L, A, \Theta, D, S)$-local Morse quasigeodesic} in $X$ if for all $t\in I$, the subpath $p|_{[t, t+S]}$ is an $(L,A,\Theta, D)$-Morse quasigeodesic. 
\end{definition}

In \cite[Theorem 7.18]{KLP}, Kapovich, Leeb and Porti show that for $L,A,\Theta,\Theta'$, $D$ with $\Theta \subset \text{int}(\Theta')$, there exist $S,L',A',D'$ such that every $(L,A,\Theta,D,S)$-local Morse quasigeodesic in $X$ is an $(L',A',\Theta',D')$-Morse quasigeodesic.

\subsection{Limit set}
We stick to the notations of Section \ref{Morse}. 
Let $\Gamma$ be a nonelementary discrete subgroup of $G$. 
The \emph{geometric limit set} $\Lambda_\Gamma$ of $\Gamma$ is defined by $$\Lambda_\Gamma=\overline{\Gamma \cdot x}\cap \partial_\infty X.$$
An isometry $g\in G$ is said to be an \emph{axial isometry} if the displacement function $d_g : X \rightarrow \mathbb R$ defined by $$d_g(x)=d_X(x,gx)$$ has a positive minimum value in $X$.
The limit $g^+=\lim_{n\rightarrow \infty}g^n\cdot x$ is called the \emph{attractive fixed point} of $g$. For more details, we refer to \cite[Section 1.9]{Eb} and \cite{Link}.
It is well known that the set of attractive fixed points of axial isometries in $\Gamma$ is dense in $\Lambda_\Gamma$ (See \cite{Be,Link}).
{\kim
\begin{definition}[Regular subgroup,  \cite{KLP2}]
A subgroup $\Gamma\subset G$ is \emph{$\tau_{mod}$-regular} if all sequences $\gamma_n\to \infty$ in $\Gamma$ are $\tau_{mod}$-regular. Furthermore, $\Gamma$ is said to be \emph{uniformly $\tau_{mod}$-regular} if every limit point of $\Gamma$ is $\tau_{mod}$-regular. We say that a representation of a group into $G$ is \emph{(uniformly) $\tau_{mod}$-regular} if its image group in $G$ is (uniformly) $\tau_{mod}$-regular.
\end{definition}

\begin{remark}\label{rem:regiff}
Let $K$ be the stabilizer of $x\in X$ in $G$ and $\mu:G\to \bar{\mathfrak a}^+$ the Cartan projection with respect to the Cartan decomposition $K(\exp \bar{\mathfrak a}^+)K$. Then the Euclidean model of Weyl chamber $\Delta$ is identified with $\bar{\mathfrak{a}}^+$ and $d_\Delta(x,gx)=\mu(g)$, and $d(d_\Delta(x, gx), V(0,\sigma_{mod}-\mathrm{int}(\sigma_{mod})))$ is the distance of $\mu(g)$ from the boundary $\partial \bar{\mathfrak{a}}^+$ of $\bar{\mathfrak{a}}^+$.
Thus the $\sigma_{mod}$-regularity of a sequence $(g_n)$ in $G$ is equivalent to the condition that the distance of $\mu(g_n)$ from $\partial \bar{\mathfrak{a}}^+$ converges to infinity.
For more details on $d_\Delta$, see \cite{KLM}.
\end{remark}
}

\subsection{Positive representations}\label{2.2}
{\kim Let $B^+$ be the set of upper triangular matrices and $B^-$ be the set of lower triangular matrices in $\pgl$. These two sets $B^+$ and $B^-$ are opposite Borel subgroups of $\pgl$ i.e., their intersection $B^+\cap B^-$ is a maximal torus of $\pgl$. 
A \emph{full flag} in $\R^m$ is a family $F$ of nested linear subspaces $$\{0\}=F^{(0)}\subset F^{(1)}\subset \cdots \subset F^{(n-1)}\subset F^{(m)}=\R^m$$ where each $F^{(i)}$ has dimension $i$.
The set of all full flags in $\R^m$ is called a \emph{full flag variety} of $\R^m$, denoted by $\mathcal F(\R^m)$. It has been well known that the full flag variety of $\R^m$ is parametrized by $\pgl/B^+$. We may therefore think of $\pgl/B^+$ as the set of all full flags.

A real matrix is called \emph{totally positive} if all its minors are positive. An upper triangle matrix is called \emph{totally positive} if all its minors which are not identically zero are positive.
Denote the set of unipotent matrices in $B^+$ by $U^+$ and the set of totally positive elements of $U^+$ by $U^+(\mathbb R_{>0})$. 
The theory of totally positive matrices was developed in 1930's and it is generalized to arbitrary semisimple real Lie groups by Lusztig \cite{Lu}.

 A configuration of full flags $(F_1,\ldots,F_n)$ is \emph{positive} if under the action of $\pgl$, it is equivalent to
$$(B^+,B^-, u_1\cdot B^-, (u_1u_2)\cdot B^-,\ldots, (u_1\cdots u_{n-2})\cdot B^- )$$ where $u_i\in U^+(\mathbb R_{>0})$ for all $i$.}
The positive configurations of full flags in $\R^m$ have a simple geometric description. A curve in $\mathbb R\mathbb P^{m-1}$ is \emph{convex} if any hyperplane intersects it in no more than $n$ points.
Fock and Goncharov prove that a configuration of $n$ real flags $(F_1,\ldots,F_n)$ in $\mathbb R\mathbb P^{m-1}$ is positive if and only if there exists a smooth convex curve $\xi$ in $\mathbb R\mathbb P^{m-1}$ such that
the flag $F_i$ is an osculating flag at a  point $x_i\in \xi$ and the order of the points $x_1,\ldots,x_n$ is compatible with an orientation of $\xi$. For more details, we refer the reader to \cite[Theorem 1.3]{FG}.

Let $\Sigma$ be a compact, connected, orientable surface (possibly with boundary) of negative Euler characteristic.
A finite volume hyperbolic metric on $\Sigma$ is \emph{admissible} if the completion of $\Sigma$ is a surface with totally geodesic boundary components and cusps.
The boundary at infinity $\partial_\infty \pi_1(\Sigma)_l$ of $\Sigma$ with $l$ cusps is the boundary at infinity of the universal cover $\widetilde \Sigma$ of $\Sigma$ equipped with some admissible hyperbolic metric on $\Sigma$. Then the set  $\partial_\infty \pi_1(\Sigma)_l$ has a cyclic ordering on points depending on the orientation of $\Sigma$.
For the representation $\rho$ of an admissible hyperbolic metric on $\Sigma$, the set $\partial_\infty \pi_1(\Sigma)_l$ is identified with the limit set of $\rho(\pi_1(\Sigma))$.
For more details, see \cite[Section 2]{LM}.

A continuous map $$\xi:\partial_\infty\pi_1(\Sigma)_l \rightarrow \mathcal F(\R^m)$$ is \emph{positive} if for any positively orientable $n$-tuple in $\partial_\infty\pi_1(\Sigma)_l$, $(x_1,\ldots,x_n)$, its image $(\xi(x_1),\ldots,\xi(x_n))$  is a positive configuration of flags. A representation $\rho:\pi_1(\Sigma)\rightarrow \pgl$ is said to be \emph{positive} if there exists a positive $\rho$-equivariant continuous map $\xi:\partial_\infty\pi_1(\Sigma)_l \rightarrow  \mathcal F(\R^m)$ for some $l$. {\kim Fock and Goncharov proved that the following properties hold for positive representations $\rho:\pi_1(\Sigma) \rightarrow \pgl$. See Theorem 1.9 and 1.10 in \cite{FG}.}
\begin{enumerate}
\item $\rho$ is discrete and faithful.
\smallskip

\item $\rho(\gamma)$ is positive hyperbolic for any non-peripheral loop $\gamma$, i.e., {\kim conjugate to a diagonal matrix with all positive eigenvalues.}
    \end{enumerate}
By (2), every $\rho(\gamma)$ has attracting and repelling fixed points in $\mathcal F(\R^m)$.

\subsection{Frenet curves}\label{sec:frenet}
{\kim A continuous curve $\xi :S^1\to \mathcal F(\R^m)$ is called a \emph{Frenet curve} if the following conditions hold.}
\begin{itemize}
\item For every pairwise distinct points $(x_1,\ldots,x_k)$ in $S^1$ and positive integers $(n_1,\ldots,n_k)$ such that $$\sum_{i=1}^k n_i \leq n,$$  the sum $$\xi^{(n_1)}(x_1)+\cdots+\xi^{(n_k)}(x_k)$$ is direct.
\item For every $x$ in $S^1$ and positive integers $(n_1,\ldots,n_k)$ such that $$l=\sum_{i=1}^k n_i \leq m,$$ we have \begin{eqnarray}\label{hyperconvex} \lim_{\substack{(y_1,\ldots,y_k) \rightarrow x, \\ y_i \ \text{all distinct}}} \left( \bigoplus_{i=1}^k \xi^{(n_i)}(y_i) \right)=\xi^{(l)}(x).\end{eqnarray}
\end{itemize}

{\kim There is a relation between Frenet curves and positive representations. For every positive representation $\rho:\pi_1(\Sigma)\to \pgl$, the positive $\rho$-equivariant continuous map associated to $\rho$ is the restriction of a Frenet curve. For more details, we refer the reader to \cite[Section 1.10--1.11 and Section 7.6--7.9]{FG} and \cite{LM}.}

\section{Regularities of convex projective structures}\label{convex}

In this section, we will prove the $\sigma_{mod}$-regularities of convex projective structures, which is a key part in proving that every convex projective structure is $\sigma_{mod}$-primitive stable.
We here introduce a good method to deal with this issue by just looking at the shape of the boundaries of convex projective domains.

Let $\Sigma$ be a compact, connected, orientable surface (possibly with boundary) of negative Euler characteristic. A convex projective structure on $\Sigma$ is a representation of $\Sigma$ as a quotient $\Omega / \Gamma$ where $\Omega$ is a convex domain in $\mathbb R\mathbb P ^2$ and $\Gamma$ is a discrete subgroup of $\mathrm{PGL}(3,\mathbb R)$ acting properly and freely on $\Omega$. 
If $\Omega$ is not a triangle, the elements of $\rsl(3,\mathbb R)$ acting on  a properly convex domain are classified as follows (see \cite{Marquis}):




\smallskip

\begin{enumerate}
\item {\bf Hyperbolic :} the matrix is conjuagate to $$\begin{bmatrix} 
           \lambda^+ & 0 & 0 \\
            0              & \lambda^0 & 0\\
            0           & 0 & \lambda^-\end{bmatrix}, \begin{array}{l}  \text{where } \lambda^+>\lambda^0>\lambda^->0 \\ \text{and } \lambda^+\lambda^0\lambda^-=1.\end{array}$$
\item {\bf Quasi-hyperbolic :} the matrix is conjugate to $$\begin{bmatrix} 
           \alpha & 1 & 0 \\
            0              & \alpha & 0\\
            0           & 0 & \beta \end{bmatrix},  \begin{array}{l}  \text{where } \alpha, \beta>0, \alpha^2 \beta=1 \\ \text{and }\alpha, \beta \neq 1.\end{array}$$
\item {\bf Parabolic :}  the matrix is conjugate to  $$\begin{bmatrix} 
           1 & 1 & 0 \\
            0              & 1 & 1\\
            0           & 0 & 1\end{bmatrix}.$$
\item {\bf Elliptic :}  the matrix is conjugate to $$\begin{bmatrix} 
           1 & 0 & 0 \\
            0              & \cos\theta & -\sin\theta \\
            0           & \sin\theta & \cos\theta\end{bmatrix},\ \text{where } 0<\theta<2\pi.$$
\end{enumerate}

\smallskip

It is said that a convex projective structure on $\Sigma$ has \emph{geodesic boundary} if the holonomy of each boundary component is hyperbolic. 
Goldman \cite{Gol} showed that the space $\mathcal T(\Sigma)$ of convex real projective structures on $\Sigma$ with geodesic boundary is of dimension $-8\chi(\Sigma)$. 
Indeed, the holonomy representations of convex projective structures with geodesic boundary are all $\sigma_{mod}$-Anosov representations. Hence they are $\sigma_{mod}$-uniformly regular \cite{KLP2}.
The natural question arises whether the other convex projective structures are $\sigma_{mod}$-uniformly regular. In order to answer the question, we consider two cases, whether the holonomy representation contains any quasi-hyperbolic element or not.

\subsection{Strictly convex case}
Suppose that $\Omega$ is a strictly convex domain in $\mathbb{RP}^m$ with $C^1$ boundary equipped with a Hilbert metric. There is a notion of parallel transport $T^t$ due to Foulon \cite{Fu1, Fu2}. 
Let $H\Omega=(T\Omega\setminus \{0\})/\mathbb R^*_+$.
The parallel Lyapunov exponent of $v\in T_x\Omega$ along $\phi^t(x,[\psi])$ is defined to be $$\eta((x,[\psi]),v)=\lim_{t\ra\infty}\frac{1}{t}\ln F(T^t(v))$$ 
where $\phi^t$ is the geodesic flow and $(x,[\psi])\in H\Omega$.

Let $\phi^t$ be a $C^1$ flow on a Riemannian manifold $W$. A point $w\in W$ is said to be \emph{regular} if there exists
a $\phi^t$-invariant decomposition
$$TW=E_1\oplus \cdots \oplus E_p$$ along $\phi^t w$ and real numbers
$$\chi_1(w)<\cdots< \chi_p(w),$$ such that, for any vector $Z_i\in E_i\setminus \{0\}$,
$$\lim_{t\ra\pm \infty}\frac{1}{t}\log ||d\phi^t(Z_i)||=\chi_i(w),$$ and
\begin{eqnarray}\label{determ}
\lim_{t\ra\pm \infty}\frac{1}{t}\log |\text{det} d\phi^t|=\sum_{i=1}^p\text{dim}E_i \cdot \chi_i(w).
\end{eqnarray}
The numbers $\chi_i(w)$ associated with a regular point $w$ are called the Lyapunov exponents of the flow at $w$.
Crampon \cite{Cr} showed that a point $w=(x,[\psi])\in H\Omega$ is regular if and only if there exists a decomposition
$$T_x\Omega=\R \psi\oplus E_0(w)\oplus (\oplus_{i=1}^n E_i(w))\oplus E_{n+1}(w),$$ and real numbers
$$-1=\eta_0(w)<\eta_1(w)<\cdots <\eta_n(w)< \eta_{n+1} = 1,$$
such that
for any $v_i\in E_i(w)\setminus \{0\}$,
$$\lim_{t\ra\pm\infty}\frac{1}{t}\ln F(T^t_w(v_i))=\eta_i(w),$$ and
$$\lim_{t\ra\pm\infty}\frac{1}{t}\ln |\text{det} T^t_w|=\sum_{i=0}^{n+1} \text{dim} E_i(w) \eta_i(w).$$

The Lyapunov exponents have to do with the convexity of the boundary of $\Omega$. For 2-dimensional $\Omega$, the boundary $\partial \Omega$ can be written as the graph of a convex function $f$ around $p\in\partial\Omega$ with $p=0, f(0)=0$. Such a function $f$ is said to be {\it approximately $\alpha$-regular} for an $\alpha\in [1,\infty]$, if
$$\lim_{t\ra 0} \frac{\ln \frac{f(t)+f(-t)}{2}}{\ln |t|}=\alpha.$$ This quantity is invariant under affine and projective
transformations. Approximately $\alpha$-regularity means that the function behaves like $|t|^\alpha$ near the origin. The case of $\alpha=\infty$ means that the boundary point belongs to a flat segment.

Let $p\in\partial\Omega$. Let $w=(x,[\psi])\in H\Omega=T\Omega\setminus \{0\}/\R_+^*$ be  regular  such that  $\psi(\infty)=p$. Let $\cal H_w$ be a horocycle based at $p$ and passing through $x$.
It is shown in \cite[Theorem 4.2]{Cr} that for any $v(w)\in T_x\cal H_w$,
\begin{eqnarray}\label{boundary}
\eta(w,v(w))=\frac{2}{\alpha(p)}-1.
\end{eqnarray}

Suppose $\gamma$ is a hyperbolic isometry whose eigenvalues are $\lambda_1>\lambda_2>\lambda_3$. Then it is shown in \cite[Section 3.6]{Crampon} that
$$\eta(w, v(w))=-1+2 \frac{\ln\frac{\lambda_1}{\lambda_2}}{\ln\frac{\lambda_1}{\lambda_3}}, $$ hence 
\begin{eqnarray}\label{alpha}
\alpha(\gamma^+)^{-1}=\frac{\ln\frac{\lambda_1}{\lambda_2}}{\ln\frac{\lambda_1}{\lambda_3}}.\end{eqnarray}  Here $w=(x,[\phi]),\ \phi(\infty)=\gamma^+$ and $v(w)\in T_x\cal H_w$.

\begin{proposition}\label{uniform-regular}
Let $\Sigma$ be a compact, connected, orientable surface possibly with boundary. 
Let $\rho : \pi_1(\Sigma) \ra \rsl(3,\mathbb R)$ be the holonomy of a strictly convex projective structure on the interior of $\Sigma$. Then $\rho$ is uniformly $\sigma_{mod}$-regular.
\end{proposition}

\begin{proof}
{\kim 
Before giving a proof, we first remark that there is a strictly convex domain $\Omega$ with $C^1$-boundary which is invariant under $\rho(\pi_1(\Sigma))$ by the following reason. The strict convexity of $\rho$ implies that the holonomy of each boundary component of $\Sigma$ is either hyperbolic or parabolic. Then, by doubling the convex projective surface associated to $\rho$, we obtain a properly convex projective surface $S$ of finite volume.
Then due to the works of Benoist \cite{Be2} and Marquis \cite{Marquis}, the properly convex domain $\Omega$ associated to $S$ is strictly convex and moreover $\partial \Omega$ is $C^1$.
By the doubling construction (see for instance \cite[Section 3.10]{Gol} or \cite[Section 9.2.2]{LM}), $\rho(\pi_1(\Sigma))$ clearly preserves $\Omega$. Therefore $\Omega$ is the desired domain.


}

We claim that $\alpha$ is bounded on $\partial \Omega$. 
The proof of the following lemma was indicated by Yi Huang. The authors thank him for pointing out references for the proof.
\begin{lemma}\label{lem:3.2} If a strictly convex $\Omega$ with $C^1$ boundary admits a quotient which is a strictly convex real projective surface possibly with cusps, then
$$\sup_{p\in\partial\Omega}\alpha(p)<\infty.$$
\end{lemma}
\begin{proof}
Benoist-Hulin  showed that the Blaschke metric is negatively curved (Proposition 3.3 in \cite{BH2})
and approaches a negative constant deep into a cusp (Proposition 3.1 in \cite{BH}), hence the curvature is pinched negative on the surface. Since the Hilbert metric and the Blaschke metric are comparable \cite[Corollary 4.7]{BH2}, $\Omega$ equipped with the Hilbert metric is Gromov hyperbolic. Furthermore Benoist \cite{BIHES} showed that $\Omega$ is quasisymmetrically convex, hence $\partial \Omega$ is $\beta$-convex for some $\beta\in [2,\infty)$ (Corollary 1.5 in \cite{BIHES}).
Benoist's notion of $\beta$-convexity is as follows. Let $f$ be a $C^1$-convex function. Denote $D_z(h)=f(z+h)-f(z)-f'(z)h$. Then $f$ is $\beta$-convex if $$\inf_{\{(z,h):h\neq 0\}} |h|^{-\beta}D_z(h)>0$$ and quasisymmetrically convex if there exists $H\geq 1$ such that $D_z(h)\leq H D_z(-h)$.

In our case, since $\partial \Omega$ is $C^1$-convex, for each $p\in \partial \Omega$, $\partial\Omega$ near $p$ can be represented by the graph of a quasisymmetrically convex $C^1$ function $f$ with $f(p)=0$.

Suppose  $\partial\Omega$ is approximately $\alpha$-regular at $p$.  Then it is easy to see that for any $\epsilon>0$, and small $|t|$ (\cite[Lemma 4.1]{Cr}),
$$|t|^{\alpha+\epsilon}\leq \frac{f(p+t)+f(p-t)}{2} \leq |t|^{\alpha-\epsilon}.$$
Hence 
$$ f(p+t)+f(p-t)= D_p(t)+ D_p(-t)\leq 2|t|^{\alpha-\epsilon}.$$ By the $\beta$-convexity of $\partial\Omega$, 
$D_p(-t)> c |t|^\beta$ for small $t$ and for some fixed constant $c>0$. Consequently, for any small $\epsilon>0$ and small $|t|$
$$2c|t|^{\beta}\leq D_p(t)+ D_p(-t) \leq 2|t|^{\alpha-\epsilon} .$$ 
This is possible only when $\alpha\leq \beta$. Hence the claim follows.
\end{proof}

Let $\gamma$ be a hyperbolic element with eigenvalues $\lambda_1>\lambda_2>\lambda_3>0$ and, the attracting fixed point $\gamma^+\in \partial\Omega$ and the repelling fixed point $\gamma^-\in \partial \Omega$. By (\ref{alpha}) and Lemma \ref{lem:3.2},  
\begin{equation}\label{eqn:l1}
\alpha(\gamma^+)=\frac{\ln\frac{\lambda_1}{\lambda_3}}{\ln\frac{\lambda_1}{\lambda_2}}\leq \beta.
\end{equation}
Considering $\gamma^{-1}$, we also obtain 
\begin{equation}\label{eqn:l2}
\alpha(\gamma^-)=\frac{\ln\frac{\lambda_1}{\lambda_3}}{\ln\frac{\lambda_2}{\lambda_3}}\leq \beta.
\end{equation}
Putting that $a_1=\ln \lambda_1 - \ln \lambda_2$ and $a_2=\ln \lambda_2 - \ln \lambda_3$, the positive Weyl chamber $\bar{\mathfrak{a}}^+$ is identified with $\{(a_1,a_2) \in \mathbb R^2 \ | \ a_1\geq 0 \text{ and } a_2\geq 0\}$. Then it holds that $$\frac{a_1+a_2}{a_1} \leq \beta \text{ and } \frac{a_1+a_2}{a_2} \leq \beta.$$
From these inequalities, we have that $$2\beta\geq  2+ \frac{a_1}{a_2}+\frac{a_2}{a_1} \geq 4 \text{ and thus } \beta\geq 2.$$
Furthermore, $$(\beta-1)^{-1} a_1 \leq a_2 \leq  (\beta-1)a_1.$$
The above inequalities imply that the set of attractive fixed points of axial isometries in $\Gamma$ is contained in a compact subset $\Theta$ of $\mathrm{int}(\sigma_{mod})$.
{\kim By the work of Benoist \cite{Benoist} (see also \cite{CG,Link})}, the geometric limit set of $\Gamma$ is the closure of the set of attractive fixed points of regular axial isometries in $\Gamma$. Therefore, $\rho$ is uniformly $\sigma_{mod}$-regular.
\end{proof}

{\kim The \emph{Hilbert length} $\ell(g)$ of a hyperbolic element $g\in \mathrm{PGL}(3,\R)$ is defined by $$\ell(g)=\ell_1(g)+\ell_2(g)=\ln \lambda_1 -\ln \lambda_3$$
where $\lambda_1>\lambda_2>\lambda_3>0$ are the eigenvalues of $g$ and $\ell_i(g)=\ln \lambda_i -\ln \lambda_{i+1}$ for $i=1,2$.
By (\ref{eqn:l1}) and (\ref{eqn:l2}), the following corollary is immediate.
}
\begin{corollary}
{\kim Let $\Sigma$ be a compact, connected, orientable surface possibly with boundary. 
Let $\rho : \pi_1(\Sigma) \ra \rsl(3,\mathbb R)$ be the holonomy of a strictly convex projective structure on the interior of $\Sigma$. Then there exists  $\beta>0$ such that for every  hyperbolic element $\rho(\gamma)$, $\gamma\in \pi_1(\Sigma)$,
$$ \ell(\rho(\gamma))\leq \beta \ell_i(\rho(\gamma))$$ for $i=1,2$. }
\end{corollary}
This is also proved in \cite[Theorem 1.23]{HS}.

\subsection{Properly convex case}
When $\Omega$ is properly convex but not strictly convex, the holonomy of one of boundary components is quasi-hyperbolic.
Recall that a quasi-hyperbolic element is conjugate to 
$$\begin{bmatrix} 
           \alpha & 1 & 0 \\
            0              & \alpha & 0\\
            0           & 0 & \beta \end{bmatrix},  \begin{array}{l}  \text{where } \alpha, \beta>0, \alpha^2 \beta=1 \\ \text{and }\alpha, \beta \neq 1.\end{array}$$

The axis of a quasi-hyperbolic isometry, which is a segment connecting the eigenvectors $p^-, p^+$ corresponding to $\alpha,\beta$ eigenvalues, can lie on $\partial \Omega$. In such a case, $\Omega$ is not strictly convex. The line corresponding to the eigenvalue $\alpha$ is outside $\Omega$.
This line is one of two tangent lines of $\Omega$ at $p^+$. The tangent line of $\Omega$ at $p^-$ contains the axis of the quasi-hyperbolic element. For more details, see \cite{Marquis}.  By this reason, the limit curve from $\partial_\infty \pi_1(\Sigma)$ to the flag variety does not satisfy the positivity and antipodality at two points $p^\pm$. Hence the holonomy representation is neither positive nor Anosov.

To deal with the properly convex case, we need a different approach as the strictly convex case.
Throughout this section, we denote $\mathrm{PGL}(3,\mathbb R)$ by $G$ for simplicity.
Embed $G$ into $\mathbb{P}(\mathrm{End} (\mathbb R^3))$ and let $\bar G$ be the closure of  $G$ in $\mathbb{P}(\mathrm{End} (\mathbb R^3))$. Then $\bar G$ is a compactification of $G$. Note that the rank of any matrix on the boundary $\partial G$ of $\bar G$ is either $1$ or $2$. We say that a sequence $(g_n)$ of $G$ is a \emph{rank $1$ sequence} if $(g_n)$ converges to a matrix of rank $1$ in $\bar G$. Similarly a sequence in $G$ is said to be a \emph{rank $2$ sequence} if it converges to a matrix of rank $2$ in $\bar G$.

\begin{lemma}\label{lem:rank2}
A discrete subgroup $\Gamma$ of $\mathrm{PGL}(3,\mathbb R)$ is $\sigma_{mod}$-regular if and only if there are no rank $2$ sequences in $\Gamma$.
\end{lemma}
\begin{proof}
Assume that $\Gamma$ is $\sigma_{mod}$-regular. As explained in Remark \ref{rem:regiff}, if we normalize $K$ to be the stabilizer group of $x$ and write the Cartan decomposition as $G=K\exp \overline{\mathfrak a}^+ K$, then $d_\Delta(x, \gamma x)$ can be identified with  the Cartan projection $\mu(\gamma)\in \bar{\mathfrak{a}}^+$ for $\gamma \in \Gamma$.  Given a sequence $(\gamma_n) \rightarrow \infty$ in $\Gamma$, the sequence $\mu(\gamma_n)$ of the Cartan projections of $\gamma_n$'s is $\sigma_{mod}$-regular. This means that if we write $\mu(\gamma_n)=\mathrm{Diag}(a_n, b_n, c_n)$ with $a_n \geq b_n \geq c_n$, then $$\lim_{n \rightarrow \infty}(a_n-b_n)=\lim_{n\rightarrow \infty} (b_n-c_n) =\infty$$ since the difference of the coordinates are distances from the walls of the Weyl chamber.
The Cartan decomposition of $\gamma_n$ is written as $\gamma_n=s_n \exp({\mu(\gamma_n)})t_n$ for some $s_n, t_n \in K$.
By passing to a subsequence we may assume that the sequences $(s_n)$ and $(t_n)$ converge to $s_\infty$ and $t_\infty$ in $K$ respectively.
Then $$\lim_{n\rightarrow \infty}e^{-a_n}s_\infty^{-1}\gamma_n t_\infty^{-1}$$ converges to a rank $1$ matrix. Thus there are no rank $2$ sequences in $\Gamma$.

Conversely assume that $\Gamma$ is not $\sigma_{mod}$-regular. Then there exists a sequence $(\gamma_n)$ in $\Gamma$ such that $(a_n-b_n)$ is uniformly bounded. By a similar argument as above, one can prove that $(e^{-a_n}\gamma_n)$ converges to a matrix of rank $2$. Therefore the converse direction is proved.
\end{proof}



We learned Lemma \ref{lem:rank2} from M. Kapovich during his visit to KIAS.
Before we prove the regularities of properly convex projective structures, we recall the following fact due to Benz\'ecri \cite{Ben} a half century ago.

\begin{theorem}[Benz\'ecri]\label{Benzecri}Let $\Omega$ be a properly convex domain in $\mathbb{RP}^{m-1}$. Suppose that a sequence $(g_n)$ in $Aut(\Omega)$ converges to a projective transformation $g_\infty$ in $\mathbb{P}(\mathrm{End} (\mathbb R^{m}))$. Then $g_\infty(\Omega)$ is a face $F$ on $\partial\Omega$ and the range of $g_\infty$ is the subspace generated by $F$. For any compact set $Z$ in the complement of the kernel of $g_\infty$, $g_n(Z)$ uniformly converges to $g_\infty(Z)$. Furthermore the kernel of $g_\infty$ has empty intersection with $\Omega$.
\end{theorem}

Now we give a proof for the $\sigma_{mod}$-regularities of properly convex projective structures.

\begin{proposition}\label{prop:regular} Let $\Sigma$ be a compact, connected, orientable surface with boundary and negative Euler characteristic. 
Every convex projective structure on the interior of $\Sigma$ is $\sigma_{mod}$-regular.\end{proposition}
\begin{proof}
Let $\rho:\pi_1(\Sigma)\ra \rsl(3,\mathbb R)$ be the holonomy representation of a convex projective structure on the interior of $\Sigma$ and $\Omega$ be an invariant convex domain. 
Let $(\gamma_n)$ be an infinite sequence in $\rho(\pi_1(\Sigma))$ which converges to $\gamma_\infty$ in $\mathbb{P}(\mathrm{End} (\mathbb R^3))$. 
Then due to Theorem \ref{Benzecri}, $\gamma_\infty(\Omega)$ is a face on $\partial \Omega$. If there are no $1$-dimensional faces on $\partial \Omega$,  $\gamma_\infty(\Omega)$ must be a point on $\partial \Omega$. It means that $\gamma_\infty$ is a matrix of rank $1$ and thus we are done by Lemma \ref{lem:rank2}. Suppose that $\partial \Omega$ has $1$-dimensional faces and $\gamma_\infty(\Omega)$ is a $1$-dimensional face $J$ on $\partial \Omega$.
Note that $J$ must equal to some conjugate image of the axis of $\rho(b)$ for some boundary component $b$ of $\Sigma$ and moreover $\rho(b)$ is either hyperbolic or quasi-hyperbolic. Let $\gamma$ be the (quasi)-hyperbolic element translating along $J$.

Choose a connected compact subset $C$ of $\Omega$ with nonempty interior. Since $C$ is a Zariski-dense subset of $\mathbb{RP}^2$, 
it suffices to show that $\gamma_n(C)$ converges to a point on $\partial\Omega$ due to Lemma \ref{lem:rank2} or Corollary \ref{cor6.3}.
According to Benz\'ecri's theorem, $C$ is contained in the complement of the kernel of $\gamma_\infty$ and thus $\gamma_n(C)$ uniformly converges to $\gamma_\infty(C)$. Clearly, $$\gamma_\infty(C) \subset \gamma_\infty(\Omega)=J.$$
We claim that $\gamma_\infty(C)$ cannot contain any interior point of $J$. If $\gamma_\infty(C)$ contains an interior point of $J$, there exists a point $p$ of $C$ such that $\gamma_\infty(p)$ is an interior point of $J$
and $\gamma_n(p)$ converges to $\gamma_\infty(p)$.
However this is impossible since any orbit of $p$ under the action of $\rho(\pi_1(\Sigma))$ can never converge to any interior point of $J$ by considering the action of $\Gamma$ on $J$ by the following reason: Choose a fundamental domain $D$ of $\rho(\pi_1(\Sigma))$ in $\Omega$ whose one end is a segment of $J$ containing $\gamma_\infty(p)$, then $\gamma_n(p)$ can be contained in $D$ only once. Hence $\gamma_n(p)$ cannot converge to $\gamma_\infty(p)$. Therefore $\gamma_\infty(C)$ can only contain two endpoints of $F$.
Since $C$ is connected, $\gamma_\infty(C)$ must be only one endpoint of $J$, which completes the proof.
\end{proof}



Even though all convex projective structures are $\sigma_{mod}$-regular, it is still not clear whether they are uniformly $\sigma_{mod}$-regular or not. We will answer this question.
Recall that the limit cone $\mathcal L_\Gamma$ of a discrete subgroup $\Gamma$ of a semisimple Lie group is defined as the smallest
closed cone in $\bar{\mathfrak a}^+$ containing  the image of the Lyapunov projection $\lambda : \Gamma \ra \bar{\mathfrak a}^+$ which is induced by the Jordan decomposition.
Benoist \cite{Be} showed that if $\Gamma$ is Zariski-dense, its limit cone is convex and invariant under the opposite involution of $\bar{\mathfrak a}^+$. Moreover the geometric limit set of $\Gamma$ in any Weyl chamber at infinity, if nonempty, is naturally identified with the set of directions in $\mathcal L_\Gamma$. 

\begin{lemma}\label{quasihyperbolic}
Let $\rho:\pi_1(\Sigma)\ra \rsl(3,\mathbb R)$ be the holonomy of a convex projective structure on a surface $\Sigma$.
If $\rho(\pi_1(\Sigma))$ has quasi-hyperbolic element and is Zariski dense, then its limit cone is $\bar{\mathfrak{a}}^+$.
\end{lemma}
\begin{proof}
Recall that a quasi-hyperbolic element $g$ is conjugate to 
 $$\begin{bmatrix} 
           \alpha & 1 & 0 \\
            0              & \alpha & 0\\
            0           & 0 & \beta \end{bmatrix},  \begin{array}{l}  \text{where } \alpha, \beta>0, \alpha^2 \beta=1 \\ \text{and }\alpha, \beta \neq 1.\end{array}$$
Since the Jordan decomposition of the above matrix is 
 $$\begin{bmatrix} 
           \alpha & 0 & 0 \\
            0              & \alpha & 0\\
            0           & 0 & \beta \end{bmatrix}
\begin{bmatrix} 
           1 & \frac{1}{\alpha} & 0 \\
            0              & 1& 0\\
            0           & 0 & 1 \end{bmatrix},$$
the Lyapunov projection $\lambda(g)$ of $g$ points to a singular direction in $\sigma_{mod}$.
Since the limit cone is convex and invariant under the opposite involution, it follows that the limit cone $\mathcal L_\rho$ of $\rho(\pi_1(\Sigma))$ is $\bar{\mathfrak{a}}^+$.
\end{proof}

Lemma \ref{quasihyperbolic} implies that any convex projective structure with a quasi-hyperbolic element is not uniformly $\sigma_{mod}$-regular.
This is indeed due to the orbit of a quasi-hyperbolic element. The orbit of a quasi-hyperbolic element is $\sigma_{mod}$-regular but converges to a singular direction. It is possible to see this by a direct computation as follows:
Let $g$ be a quasi-hyperbolic matrix in $\mathrm{PGL}(3,\mathbb R)$. 
As mentioned above, we may assume that  $$g=\begin{bmatrix} 
           \alpha & 1 & 0 \\
            0              & \alpha & 0\\
            0           & 0 & \beta \end{bmatrix},  \begin{array}{l}  \text{where } \alpha, \beta>0, \alpha^2 \beta=1 \\ \text{and }\alpha, \beta \neq 1.\end{array}$$
By a straight computation, the Cartan projection of $g^n$ is written as 
$$\mu(g^n)= \frac{1}{2} (\ln f(n)-\ln 2, \ \ln2+4n\ln \alpha -\ln f(n), \ -4n \ln \alpha),$$
where $f(n)=n^2 \alpha^{2n-2}+2\alpha^{2n}+ \sqrt{n^4\alpha^{4n-4}+4n^2 \alpha^{4n-2}}$.
One can check that the distance between the sequence $(\mu(g^n))$ and any wall in $\bar{\mathfrak a}^+$ goes to infinity. Hence $g^n$ is $\sigma_{mod}$-regular. To find where the sequence converges, let $\theta_n$ be the angle between $\mu(g^n)$ and the singular line $x=y$ on the plane $x+y+z=0$. Then we have that $$\tan \theta_n = \frac{\ln f(n)-2n \ln \alpha - \ln 2}{2\sqrt 3 n \ln \alpha}.$$
By a computation, it can be verified that $$\lim_{n\ra\infty} \frac{\ln f(n)}{n}=2\ln \alpha.$$
Therefore it is derived that $$\lim_{n \ra \infty} \tan \theta_n =0 \text{ and thus }\lim_{n\ra\infty} \theta_n =0.$$
This implies that the sequence $(\mu(g^n))$ converges to a singular point at infinity.
Summarizing the results so far, we have Theorem \ref{cr}

\section{Regularity and Grassmannians}

In order to deal with a more general Lie group $\pgl$, we will look at the actions of regular and singular sequences of elements of $\pgl$ on Grassmannians.

\subsection{Compound matrices}
Let $I^m_k$ denote the set of strictly increasing sequences of $k
$ integers in $\{1,\ldots,m\}$. Let $A$ be an $m\times n$ real matrix. Then for each $\textbf i \in I^m_k$ and $\textbf j \in I^n_k$, define the $k\times k$ submatrix $A[\textbf i, \textbf j]$ of $A$ as $$A[\textbf i, \textbf j]=A \left[ \begin{array}{c} i_1,\ldots,i_k \\ j_1,\ldots, j_k \end{array}\right].$$
Namely, $A[\textbf i, \textbf j]$ is the submatrix of $A$ determined by the rows indexed $i_1,\ldots,i_k$ and columns indexed $j_1,\ldots,j_k$.
The $k$th \emph{compound matrix} of $A$ is defined as the $\binom{m}{k} \times \binom{n}{k}$ matrix with entries $$(\det A[\textbf i, \textbf j])_{\textbf i \in I^m_k,  \ \textbf j \in I^n_k}$$ and is denoted by $C_k(A)$ where index sets are arranged in lexicographic order. 
For an $m\times n$ matrix $A$ and $n \times l$ matrix $B$, it follows from the Cauchy-Binet formula that $$C_k(AB)=C_k(A)C_k(B)$$ for each $k \leq \min\{m,n,l\}$.

\subsection{Grassmannians and Pl\"{u}cker coordinates}\label{grassmannians}
The Grassmannian $G(m,k)$ is defined as the set of $k$-dimensional subspaces of the vector space $\mathbb R^m$. The Grassmannian $G(m,k)$ is described as a subvariety of projective space via the Pl\"{u}cker embedding $\psi : G(m,k) \rightarrow \mathbb P(\wedge^k \mathbb R^m)$ which is defined by $$\mathrm{Span}(v_1,\ldots,v_k) \mapsto [v_1\wedge\cdots\wedge v_k].$$
It is well known that the Pl\"{u}cker embedding is a well-defined map and its image is closed. Indeed it is a subvariety. Let $(e_1,\ldots,e_m)$ be the canonical ordered basis for $\mathbb R^m$. Then for any element of $\omega\in G(m,k)$, $\psi(\omega)$ has a unique representation in the form of $$\psi(\omega)= \sum_{\textbf i \in I^m_k} a_{\textbf i} e_{\textbf i}=\sum_{1\leq i_1 <\cdots<i_r\leq m} a_{i_1,\ldots,i_k}(e_{i_1}\wedge \cdots \wedge e_{i_k}).$$
The homogeneous coordinates $[a_{\textbf i}]$ are called the \emph{Pl\"{u}cker coordinates} on $\mathbb P(\wedge^k \mathbb R^m)$ for $\omega \in G(m,k)$.

Now we will look at the action of $\pgl$ on $\mathbb P(\wedge^k \mathbb R^m)$ via the Pl\"{u}cker coordinates. 
A matrix $A\in \pgl$ defines a map $f_A : \mathbb P(\wedge^k \mathbb R^m) \rightarrow \mathbb P(\wedge^k \mathbb R^m)$ by $$f_A \left(  \sum_{\textbf i \in I^m_k} a_{\textbf i} e_{\textbf i} \right) =  \sum_{\textbf i \in I^m_k} a_{\textbf i} A(e_{\textbf i})=\sum_{1\leq i_1 <\cdots<i_k\leq m} a_{i_1,\ldots,i_k}(Ae_{i_1}\wedge \cdots \wedge Ae_{i_k}).$$
By the definition of $f_A$, we have that for any $v_1, \ldots,v_k \in \mathbb R^m$, $$f_A(v_1\wedge \cdots \wedge v_k)=Av_1\wedge \cdots \wedge Av_k.$$
To look at this in terms of the Pl\"{u}cker coordinates, let $B$ be the $m \times k$ matrix with column vectors $v_1, \ldots, v_k$. Then the Pl\"{u}cker coordinate $a_{\textbf i}$ for $v_1\wedge \cdots \wedge v_k$ is the $k\times k$ minor of $B$ obtained by taking all $k$ columns and the $k$ rows with indices in $\textbf i \in I^m_k$. In other words, the $k$th compound matrix $C_k(B)$ of $B$ is exactly the Pl\"{u}cker coordinates for $v_1\wedge \cdots \wedge v_k$. In a similar way, it can be seen that the $k$th compound matrix $C_k(AB)$ is the Pl\"{u}cker coordinates for $Av_1\wedge \cdots \wedge Av_k$. It follows from the Cauchy-Binet formula that $$C_k(AB)=C_k(A)C_k(B).$$
Hence we have that $$f_A \left(\sum_{\textbf i \in I^m_k} a_{\textbf i}e_{\textbf i}\right) = \sum_{\textbf i \in I^m_k} \left( \sum_{\textbf j \in I^m_k} A(\textbf i,\textbf j) a_{\textbf j} \right) e_{\textbf i}$$ where $A(\textbf i,\textbf j)$ denotes the determinant of $A[\textbf i, \textbf j]$. In conclusion, $f_A$ is the projective linear transformation of $\mathbb P(\wedge^k \mathbb R^m)$ that is represented by $C_k(A)$.

\subsection{Dynamics of $\sigma_{mod}$-regular sequences on Grassmannians}
Let $(g_n)$ be a $\sigma_{mod}$-regular sequence in $\pgl$. We want to see the dynamics of regular sequences on Grassmannians. 
For simplicity, we first assume that $g_n=\mathrm{Diag}(\lambda_{n,1},\ldots,\lambda_{n,m})$ with $\lambda_{n,1}\geq \cdots \geq \lambda_{n,m}>0$ for each $n\in \mathbb N$. 
{\kim Let $K=\mathrm{PO}(m)$, $x=eK\in \pgl/K$ and $K\exp\overline{\mathfrak a}^+K$ be the corresponding Cartan decomposition.
The Euclidean model of Weyl chamber $\Delta$ is canonically identified with $\overline{\mathfrak a}^+$ and $d_\Delta(x,g_nx)=\mu(g_n)$. It is easy to check that $$\overline{\mathfrak a}^+=\{(a_1,\ldots,a_m)\in \R^m \ | \ a_1+\cdots +a_m=0 \text{ and } a_i\geq a_{i+1} \text{ for all }i  \},$$
and the Cartan projection $\mu(g_n)$ of $g_n$ is written by
\[ \mu(g_n)=(\ln \lambda_{n,1}-\ln s_n, \ldots,  \ln \lambda_{n,m}-\ln s_n) \]
where $s_n=\sqrt[m]{\lambda_{n,1}\cdots\lambda_{n,m}}$. 
There are the $(m-1)$ walls of $\overline{\mathfrak a}^+$. For each $i=1,\ldots,m-1$,
define the $i$th wall of $\overline{\mathfrak a}^+$ by 
\[  \partial_i \overline{\mathfrak a}^+ =\{(a_1,\ldots,a_m)\in \overline{\mathfrak a}^+ \ | \ a_i=a_{i+1}  \} \] 
Then the distance of $\mu(g_n)$ from the $i$th wall is $(\ln\lambda_{n,i}-\ln\lambda_{n,i+1})/\sqrt{2}$.
The definition of $\sigma_{mod}$-regularity is equivalent to the condition that the distance of $\mu(g_n)$ from the $i$th wall converges to infinity for all $i=1,\ldots,m-1$, i.e.
}
\begin{eqnarray}\lim_{n\rightarrow \infty} \frac{\lambda_{n,i+1}}{\lambda_{n,i}}=0 \label{eqn:regular}\end{eqnarray} for all $i=1,\ldots,m-1$. 
As described before, the action of $g_n$ on $G(m,k)$ is the projective linear transformation of $\mathbb P(\wedge^k\mathbb R^m)$ with matrix $C_k(g_n)$. Since each $g_n$ is a diagonal matrix, it can be easily seen that its $k$th compound matrix $C_k(g_n)$ is also a diagonal matrix for any $k=1,\ldots,m$. More precisely, 
$$g_n \cdot \left(\sum_{\textbf i \in I^m_k} a_{\textbf i}e_{\textbf i}\right) = \sum_{\textbf i \in I^m_k} \left( \sum_{\textbf j \in I^m_k} g_n(\textbf i,\textbf j) a_{\textbf j} \right) e_{\textbf i}=\sum_{\textbf i \in I^m_k} \lambda_{n,\textbf i} a_{\textbf i} e_{\textbf i}.$$
Here $\lambda_{n,\textbf i}= \lambda_{n, i_1} \cdots \lambda_{n, i_k}$ for $\textbf i=(i_1,\ldots, i_k) \in I^m_k$.

Let $\textbf i_1=(1,\ldots,k)$. Then it is not difficult to see that for any $\textbf i \neq \textbf i_1 \in I^m_k$, $$\lim_{n\rightarrow \infty} \frac{\lambda_{n,\textbf i}}{\lambda_{n,\textbf i_1}}=0.$$
This implies that if $a_{\textbf i_1} \neq 0$,
$$\lim_{n\rightarrow \infty}g_n \cdot \left[ \sum_{\textbf i \in I^m_k} a_{\textbf i}e_{\textbf i} \right]= \left[ \sum_{\textbf i \in I^m_k} \lambda_{n,\textbf i} a_{\textbf i} e_{\textbf i} \right]= \left[ \sum_{\textbf i \in I^m_k} \frac{\lambda_{n,\textbf i}}{\lambda_{n,\textbf i_1}} a_{\textbf i} e_{\textbf i} \right] = [e_{\textbf i_1}]$$
where $[\textbf v]$ denotes the point of $\mathbb P(\wedge^k \mathbb R^m)$ corresponding to $\textbf v \in \wedge^r \mathbb R^m$.

\begin{proposition}\label{prop:reggras}
Let $(g_n)$ be a $\sigma_{mod}$-regular sequence in $\pgl$. Then there exists a subsequence $(g_{n_i})$ such that for each $k=1,\ldots,m-1$, there is a hyperplane $H^-_k$ of $\mathbb P(\wedge^k \mathbb R^m)$ such that on the complement of $H^-_k$ the sequence of maps $(g_{n_i})$ converges pointwise to a constant map. 
\end{proposition}
\begin{proof}
Due to the Cartan decomposition of $\pgl$, each $g_n$ can be written as $g_n=s_n a_n t_n$ for some $s_n, t_n \in K$ and $a_n \in \exp \bar{\mathfrak a}^+$.
By passing to a subsequence, we may assume that $s_n$ and $t_n$ converge to $s_\infty$ and $t_\infty$ in $K$ respectively. For each $k=1,\ldots,m-1$, we put $H^-_k= t_\infty^{-1}(P_k)$ where $P_k$ is the hyperplane of $\mathbb P(\wedge^k \mathbb R^m)$ corresponding to $a_{\textbf i_1}=0$ in the Pl\"{u}cker coordinates. Then it can be easily shown that for every point $\omega \notin H^-_k$, $$\lim_{n\rightarrow \infty} g_n \cdot \omega= [s_\infty e_{\textbf i_1}].$$ This implies the proposition.
\end{proof}

We now look at the case that a sequence $(g_n)$ is not $\sigma_{mod}$-regular. 
For simplicity, as before, we first assume that every $g_n$ is a diagonal matrix given by $\mathrm{Diag}(\lambda_{n,1},\ldots,\lambda_{n,m})$  with $\lambda_{n,1}\geq \cdots \geq \lambda_{n,m}>0$.
If a sequence $(g_n)$ is not $\sigma_{mod}$-regular, then for some $ k_0 \in \{1,\ldots,m-1\}$ the property of (\ref{eqn:regular}) fails. Namely,
\begin{equation}\label{singular} \lim_{n\rightarrow \infty} \frac{\lambda_{n,k_0+1}}{\lambda_{n,k_0}}=c \text{ for some } c>0. \end{equation}
We may assume that $k_0$ is the smallest number for which (\ref{eqn:regular}) fails. Looking at the action of $g_n$ on $\mathbb P(\wedge^{k_0}\mathbb R^m)$, 
\begin{equation}\label{eqn6} \lim_{n\rightarrow \infty} \frac{\lambda_{n,\textbf i_2}}{\lambda_{n,\textbf i_1}}=\lim_{n\rightarrow \infty} \frac{\lambda_{n,k_0+1}}{\lambda_{n,k_0}}=c \end{equation}
 where $\textbf i_1=(1,\ldots,k_0)$ and $\textbf i_2 = (1,\ldots,k_0-1,k_0+1)$.
Due to $c\neq 0$, if $$\lim_{n\rightarrow \infty}g_n\cdot \omega_1 = \lim_{n\rightarrow \infty}g_n\cdot \omega_2$$ for $\omega_1, \omega_2 \in \mathbb P(\wedge^{k_0}\mathbb R^m) \setminus P_{k_0}$, then $\omega_1$ and $\omega_2$ must be contained in a hyperplane of $\mathbb P(\wedge^{k_0}\mathbb R^m)$ for which the ratio of the $a_{\textbf i_2}$-coordinate to the $a_{\textbf i_1}$-coordinate is constant. In summary we have the following proposition.

\begin{proposition}\label{regular}
Let $(g_n)$ be an infinite sequence in $\pgl$. Suppose that for each $k=1,\ldots, m-1$, there exists a {\kim basis} of $\wedge^k \mathbb R^m$ such that the limit of the sequence of maps $(g_n)$ sends all elements of the basis to one point in $\mathbb P(\wedge^k \mathbb R^m)$. Then $(g_n)$ is $\sigma_{mod}$-regular.
\end{proposition}

\begin{proof}
Let $g_n=s_n a_n t_n$ be the Cartan decomposition of $g_n$. By passing to a subsequence, we assume that $s_n$ and $t_n$ converge to $s_\infty$ and $t_\infty$ respectively.
{\kim We follow the notation of the proof of Proposition \ref {prop:reggras}.} Assume that $(g_n)$ is not $\sigma_{mod}$-regular. Then from the observation above, there is a number $0<k_0<n$ such that (\ref{eqn:regular}) fails. Furthermore
if $$\lim_{n \ra \infty} g_n \cdot \omega_1 =\lim_{n\ra \infty} g_n \cdot \omega_2$$ for $\omega_1, \omega_2 \in \mathbb P(\wedge^{k_0}\mathbb R^m) \setminus H^-_{k_0}$, {\kim as is shown just before Proposition \ref{regular}}, $t_\infty(\omega_1)$ and $t_\infty(\omega_2)$ must be contained in a hyperplane of $\mathbb P(\wedge^{k_0}\mathbb R^m)$.
This implies that any {\kim basis} of $\wedge^{k_0}\mathbb R^m$ can never converge to one point in $\mathbb P(\wedge^{k_0} \mathbb R^m)$.
This makes a contradiction to the assumption. Therefore $(g_n)$ with the property in the proposition must be $\sigma_{mod}$-regular.
\end{proof}

Proposition \ref{regular} provides a tool to check the $\sigma_{mod}$-regularity of a sequence. This will be useful in proving the $\sigma_{mod}$-regularities of positive representations later.
Note that in order to apply Proposition \ref{regular} to a sequence, we need information about the action of a sequence on each projective space $\mathbb P(\wedge^k\mathbb R^m)$. In other words it is possible to prove the $\sigma_{mod}$-regularity of a sequence only with information about the action of a sequence on each Grassmannian as follows.

\begin{corollary}\label{cor6.3}
Let $(g_n)$ be a sequence in $\pgl$. Suppose that for each $k=1,\ldots,m-1$, there exists a Zariski-dense subset of $G(m,k)$ such that the sequence of maps $g_n$ converges to a constant map on the Zariski-dense subset. Then the sequence $(g_n)$ is $\sigma_{mod}$-regular.
\end{corollary}
\begin{proof}
We think of $G(m,k)$ as a subvariety of $\mathbb P(\wedge^k\mathbb R^m)$ via the Pl\"{u}cker coordinates. Obviously, the action of $\pgl$ on $\mathbb P(\wedge^k\mathbb R^m)$ preserves the subvariety $G(m,k)$. 
Suppose that the sequence $(g_n)$ is not $\sigma_{mod}$-regular. Then for some $k_0$, equation (\ref{eqn6}) holds. We follow the notations of (\ref{eqn6}).
Let $H_y$ be a hyperplane of $\mathbb P(\wedge^{k_0}\mathbb R^m)$ for which the ratio of the $a_{\textbf i_2}$-coordinate to the $a_{\textbf i_1}$-coordinate is $y\in \mathbb R$.
It is easy to see that $G(m,{k_0})\cap H_y$ is a subvariety of $G(m,{k_0})$ with codimension $1$. 
Let $g_n=s_n a_n t_n$ be the Cartan decomposition of $g_n$. Assume that $s_n$ and $t_n$ converge to $s_\infty$ and $t_\infty$ respectively.
If $$\lim_{n\ra \infty}g_n \cdot \omega_1= \lim_{n\ra \infty}g_n \cdot \omega_2$$ for $\omega_1, \omega_2 \in G(m,{k_0})$, then $\omega_1$ and $\omega_2$ must be contained in the subvariety $t_\infty^{-1}\cdot ( G(m,{k_0})\cap H_y)=G(m,{k_0})\cap t_\infty^{-1} \cdot H_y$  of $G(m,{k_0})$ for some $y\in \mathbb R$. This means that any Zariski-dense subset of $G(m,{k_0})$ can never converge to one point in $G(m,{k_0})$. This implies the corollary.
\end{proof}

We have seen so far a sufficient condition for a sequence to be $\sigma_{mod}$-regular in terms of the Grassmannian. Applying it to positive representations, we have the following proposition.

\begin{proposition}\label{thm:regular}
Every positive representation is $\sigma_{mod}$-regular.
\end{proposition}

\begin{proof}
Let $\rho : \pi_1(\Sigma) \ra \pgl$ be a positive representation for a compact, connected, orientable surface $\Sigma$ (possibly with boundary) of negative Euler characteristic. Let $\xi :\partial_\infty \pi_1(\Sigma)_l \ra \flagv$ be the positive $\rho$-equivariant homeomorphism associated to $\rho$ for some $l$. {\kim As mentioned in Section \ref{sec:frenet}, $\xi$ is the restriction of a Frenet curve.}
Let $(\gamma_n)$ be an arbitrary sequence of elements of $\pi_1(\Sigma)$. Let $\gamma_n^+$ and $\gamma_n^-$ denote the attracting fixed point and repelling fixed point of $\gamma_n$ on $\partial_\infty \pi_1(\Sigma)_l$ respectively. Assume that $\gamma_n^+$ and $\gamma_n^-$ converge to $\gamma_\infty^+$ and $\gamma_\infty^-$ respectively. Note that $\gamma_\infty^+$ and $\gamma_\infty^-$ might be equal. Then on $\partial_\infty \pi_1(\Sigma)_l \setminus \{\gamma_\infty^-\}$, the sequence of maps $\gamma_n$ converges to the constant map, which sends all of $\partial_\infty \pi_1(\Sigma)_l \setminus \{\gamma_\infty^-\}$ to $\gamma_\infty^+$. 

Since $\xi$ is a restriction of a Frenet curve, it is possible to choose $m$ points $x_1,\ldots, x_m$ on $\partial_\infty \pi_1(\Sigma)_l \setminus \{\gamma_\infty^-\}$ so that $\{\xi^{1}(x_1),\ldots,\xi^{1}(x_m)\}$ is a basis of $\mathbb R^m$. Then for each $k=1,\ldots,m-1$, the set $$\mathcal B_k=\{\xi^{(1)}(x_{i_1}) \wedge \cdots \wedge \xi^{(1)}(x_{i_k}) \ | \ 1\leq i_1<\cdots<i_k \leq m\}$$ is a basis for $\wedge^k \mathbb R^m$. Noting that for all $i=1,\ldots,m$, $$\lim_{n\ra \infty} \gamma_n \cdot x_i = \gamma_\infty^+$$ and $\xi$ is a restriction of a Frenet curve, it follows from the property (\ref{hyperconvex}) of Frenet curve that $$\lim_{n\ra \infty} \left[\bigoplus_{j=1}^k \rho(\gamma_n) \cdot \xi^{(1)}(x_{i_j}) \right]=\lim_{n\ra \infty} \left[\bigoplus_{j=1}^k \xi^{(1)}(\gamma_n\cdot x_{i_j}) \right]=\xi^{(k)}(\gamma_\infty^+).$$ 
In other words, the sequence of maps $\rho(\gamma_n)$ sends all elements of $\mathcal B_k$ to one point $\xi^{(k)}(\gamma_\infty^+)$ for each $k=1,\ldots,m-1$.
By Proposition \ref{regular}, it immediately follows that the sequence $(\rho(\gamma_n))$ is $\sigma_{mod}$-regular. We completes the proof.
\end{proof}

We have proved the $\sigma_{mod}$-regularities of convex projective structures and positive representations.
The $\sigma_{mod}$-regularity implies discrete length spectrum as follows.

\begin{proof}[Proof of Corollary \ref{cor:1.6}] 
{\kim Let $S=\mathbb H^2/\Gamma$ be a hyperbolic surface and $\Lambda_\Gamma \subset \partial_\infty \mathbb H^2$ be the limit set of $\Gamma$.
Every positive representation $\rho:\Gamma\to \pgl$ admits a positive $\rho$-equivariant continuous map $\xi: \Lambda_\Gamma\to \mathcal F(\R^m)$.
In particular, if $\rho$ is type-preserving, then $\xi$ is indeed a homeomorphism.

Let $S_0$ be the subsurface of $S$ obtained by cutting off all cusp regions in $S$. 
Choose a fundamental domain $M_0$ in $\mathbb H^2$ corresponding to $S_0$. This can be obtained by cutting off cusp regions from a fundamental domain of $S$ in $\mathbb H^2$.
Hence $M_0$ is compact. Then for any closed geodesic $c$ on $S$,
there exists a lift $\tilde c$ to $\mathbb H^2$ such that $\tilde c\cap M_0\neq \emptyset$.
Let $\gamma_{c} \in \Gamma$ be the positive hyperbolic element corresponding to $c$ which has $\tilde c$ as the invariant axis of $\gamma_{ c}$ and is compatible with the orientation on $c$. 
Clearly, $\tilde c$ is the unique oriented  bi-infinite geodesic from $\gamma_c^-$ to $\gamma_c^+$.
Let $(\Lambda_\Gamma\times \Lambda_\Gamma)'$ denote the set of pairs of distinct points of $\Lambda_\Gamma$.
Since $\tilde c$ intersects the compact subset $M_0$ of $\mathbb H^2$ for all oriented closed geodesics $c$ on $S$, it can be easily seen that the set of all pairs $(\gamma_c^+,\gamma_c^-)$ where $c$ is an oriented closed geodesic on $S$ is a relatively compact subspace in $(\Lambda_\Gamma\times \Lambda_\Gamma)'$.

Let $p$ and $q$ be distinct points of $\Lambda_\Gamma$. Then by the positivity of $\xi$, it follows that $\xi(p)$ and $\xi(q)$ are opposite and hence there is a unique maximal flat whose ideal boundary contains $\xi(p)$ and $\xi(q)$, denoted by $P(\xi(p),\xi(q))$.
We consider the map from $(\Lambda_\Gamma\times \Lambda_\Gamma)'$ to $\mathbb R$ defined as
\[ (p,q)\mapsto d_X(x,  P(\xi(p),\xi(q))). \]
The map is continuous. By the relative compactness of the set of all pairs $(\gamma_c^+,\gamma_c^-)$ for oriented closed geodesics $c$ on $S$, there is a uniform constant $D>0$ such that for every oriented closed geodesic $c$, \[ d_X(x, P(\xi(\gamma_c^+),\xi(\gamma_c^-)))\leq D.\] 
From the $\rho$-equivariance of $\xi$ and $\gamma_c\cdot \gamma_c^\pm=\gamma_c^\pm$, 
\begin{align*} \rho(\gamma_c)\cdot P(\xi(\gamma_c^+),\xi(\gamma_c^-)) &= P(\rho(\gamma_c)\cdot \xi(\gamma_c^+),\rho(\gamma_c)\cdot \xi(\gamma_c^-)) \\
& =P(\xi(\gamma_c\cdot \gamma_c^+),\xi(\gamma_c\cdot \gamma_c^-))=P(\xi(\gamma_c^+),\xi(\gamma_c^-)).
\end{align*}
In other words, the maximal flat $P(\xi(\gamma_c^+),\xi(\gamma_c^-))$ is the maximal flat invariant under the action $\rho(\gamma_c)$ on $X$. The action of $\rho(\gamma_c)$ on $P(\xi(\gamma_c^+),\xi(\gamma_c^-))$ is a translation which involves only the eigenvalues of $\rho(\gamma_c)$.
Since $\rho(\gamma_c)$ is hyperbolic, the translation vector for the action of $\rho(\gamma_c)$ on $P(\xi(\gamma_c^+),\xi(\gamma_c^-))$ is $\lambda(\rho(\gamma_c))$ i.e.,
for any point $y \in P(\xi(\gamma_c^+),\xi(\gamma_c^-))$, \[d_\Delta(y, \rho(\gamma_c)\cdot y)=\lambda(\rho(\gamma_c)) \]
 where $\lambda : \pgl \to \overline{\mathfrak a}^+$ is the Lyapunov projection.
 Choose $\bar x_c$ in $P(\xi(\gamma_c^+),\xi(\gamma_c^-))$ such that $d_X(x, P(\xi(\gamma_c^+),\xi(\gamma_c^-)))=d_X(x, \bar x_c)$.
Taking $K=\mathrm{PO}(m)$ and $x=eK\in \pgl/K=X$, we also have $d_\Delta(x, \rho(\gamma_c)\cdot x)=\mu(\rho(\gamma_c))$. Let $\|\cdot\|$ be the Euclidean metric on $\mathfrak a$. By the triangle inequality (see for instance \cite[Remark 5.5]{KLP}) $$\|d_\Delta(x,y)-d_\Delta(x',y')\|\leq d_X(x,x')+d_X(y,y'),$$
we have
\begin{align} 
\|\mu(\rho(\gamma_c)) -\lambda(\rho(\gamma_c)) \| &= \|d_\Delta(x, \rho(\gamma_c)\cdot x)-d_\Delta(\bar x_c,  \rho(\gamma_c)\cdot \bar x_c) \| \label{eqn:app}  \\
&\leq d_X(x,\bar x_c)+d_X(\rho(\gamma_c)\cdot x,\rho(\gamma_c)\cdot \bar x_c) \nonumber \\
&=2d_X(x,\bar x_c) \leq 2D.  \nonumber
\end{align}

We are now ready to prove the corollary.
Suppose the simple $\ell_i$-spectrum for $\rho$ is not discrete for some $i$. Then there exists an infinite sequence $(c_n)$ of closed geodesics on $S$ such that 
$\ell_i(\rho(\gamma_n))$ accumulates where we write $\gamma_n=\gamma_{c_n}$ for simplicity. 
Note that $\ell_i(\rho(\gamma_n))$ is $\sqrt 2$ times of the distance of $\lambda(\rho(\gamma_n))$ from the $i$th wall of $\overline{\mathfrak a}^+$.
By (\ref{eqn:app}), the distance between $\mu(\rho(\gamma_n))$ and $\lambda(\rho(\gamma_n))$ is uniformly bounded. This implies that once the distance of $\lambda(\rho(\gamma_n))$ from the $i$th wall of $\overline{\mathfrak a}^+$ accumulates, so does the distance of $\mu(\rho(\gamma_n))$ from the $i$th wall of $\overline{\mathfrak a}^+$.
However, by the $\sigma_{mod}$-regularity of $\rho$, the distance of $\mu(\rho(\gamma_n))$ from the $i$th wall of $\overline{\mathfrak a}^+$ can not accumulate for any $i=1,\ldots,m-1$ (see Remark \ref{rem:regiff}).
Therefore we conclude that the simple $\ell_i$-spectrum for $\rho$ can not accumulate i.e. it is discrete for all $i=1,\ldots, m-1$.
}
\end{proof}

\section{Primitive stable representations}

We have seen the $\sigma_{mod}$-regularity for convex projective structures and moreover positive representations.
In this section, we prove that they are $\sigma_{mod}$-primitive stable.
Gu\'eritaud, Guichard, Kassel and Wienhard suggested the definition of primitive stable representation in higher rank in \cite[Remark 1.6]{GGKW}. Here we give a definition of primitive stable representation in terms of Morse quasigeodesics introduced by Kapovich, Leeb and Porti \cite{KLP2}, which is equivalent to the previous definitions.

Let $\Gamma$ be a non-abelian free group.
Given a representation $\rho : \Gamma \rightarrow G$ and a basepoint $x \in X$, a $\rho$-equivariant orbit map $\tau_{\rho,x} : \mathcal C(\Gamma,S) \ra X$ is defined by $\tau_{\rho,x}(w)=\rho(w)\cdot x$. 

\begin{definition}
A representation $\rho:\Gamma \rightarrow G$ is $\tau_{mod}$-\emph{primitive stable} if there exist constants $L,A,D$ and a compact set $\Theta \subset \text{ost}(\tau_{mod})$ for some face $\tau_{mod}$ of the model Weyl chamber $\sigma_{mod}$ and a basepoint $x\in X$ such that the orbit map $\tau_{\rho,x}$ takes all bi-infinite primitive geodesics to $(L,A,\Theta,D)$-Morse quasigeodesics.
\end{definition}

In hyperbolic $3$-manifold theory, there are two important results on primitive stable representations. The first is the stableness of primitive stable representations in character variety and the second is the properness of the action of the outer automorphism group of a free group on the space of primitive stable representations. These two properties are extended to higher rank symmetric spaces by combining Minsky's idea in \cite{Minsky} with the work of Kapovich-Leeb-Porti in \cite{KLP2} as follows.

\begin{theorem}\label{thm:disc} Let $\Gamma$ be a non-abelian free group and $G$  a semisimple Lie group without compact factors. Then the set $\mathcal{PS}(\Gamma,G)$ of primitive stable representations is open in the character variety of $\Gamma$ in $G$, and the action of the outer automorphism group of $\Gamma$ on $\mathcal{PS}(\Gamma,G)$ is properly discontinuous.
\end{theorem}

\begin{proof}[Sketch of proof]
For reader's convenience, we recall their works and then sketch a proof briefly.
The openness of primitive stable representations follows from \cite[Theorem 7.33]{KLP} that a local Morse quasigeodesic is a global Morse quasigeodesic. 
Let $\rho: \Gamma \rightarrow G$ be a primitive stable representation. Fix a word metric on the Cayley graph of a group $\Gamma$ and consider the orbit map for a fixed base point $x\in X$. Then there exist constants $(L,A, \Theta, D)$ such that any bi-infinite geodesic defined by a primitive element
is mapped to an $(L,A,\Theta,D)$-Morse quasigeodesic. Then for any $S>0$, every primitive bi-infinite geodesic  is mapped by the orbit map to a $(L,A, \Theta, D,S)$-local Morse quasigeodesic. 
{\kim The local Morse property for primitive bi-infinite geodesics involves only finite orbit points due to $\Gamma$-equivariance. Hence all representations sufficiently close to $\rho$ preserve the local Morse property under the relaxed Morse parameters. Then the local to global property for Morse quasigeodesics in \cite[Theorem 7.26]{KLP} implies that for all representations sufficiently close to $\rho$, any primitive bi-infinite geodesic is mapped to an $(L',A',\Theta',D')$-Morse quasigeodesic} for some Morse parameters $(L',A',\Theta',D')$, i.e., they are primitive stable. We refer the reader to \cite[Section 7]{KLP} for more detailed proof about this.

The properness of the action of the outer automorphism group of $\Gamma$ on the space of primitive stable representations follows from Minsky's idea in \cite{Minsky}. Just for completeness we give an outline.
By the definition of the primitive stability of $\rho$, there exists $r=r(\rho)>0$ such that for $w \in \Gamma$,
$$r\|w\|<t_\rho(w),$$ where $t_\rho(w)$ is a translation length of $\rho(w)$ with respect to a metric on $X$, and $$\|w\|=\inf_{g\in \Gamma}|gwg^{-1}|$$ is
the infimum of the word lengths among its conjugates with respect to a fixed generating set. By triangle inequality,
$$t_\rho(w)< R \|w\|$$ for $R$ depending on $\rho$. Hence once a compact set $C$ in the set of primitive stable representations is given, there exist uniform constants $r$ and $R$ on $C$ satisfying the above inequalities. Hence if $[\Phi]\in Out(\Gamma)$ satisfies $[\Phi](C)\cap C\neq \emptyset$, then for $[\rho]$ in this intersection
$$\|\Phi(w)\|\leq\frac{1}{r}t_\rho(\Phi(w))=\frac{1}{r}t_{\rho\circ\Phi}(w)\leq \frac{R}{r} \|w \|.$$ But it is shown that the set of such $[\Phi]$ is finite \cite{Minsky}.
\end{proof}

{\kim Recall that $\mathcal P$ is the set of all bi-infinite geodesics in the Cayley graph of $\Gamma$ lifted from $\overline w$ for all primitive elements $w\in \Gamma$.
Let $\mathcal P_e\subset \mathcal P$ denote the set of bi-infinite primitive geodesics $q:\mathbb Z \to \Gamma$ with $q(0)=e$.
Each $\overline w$ for a primitive word $w$ is lifted to a $\Gamma$-invariant family of bi-infinite geodesics in the Cayley graph.
Hence, due to the $\Gamma$-equivariance, in order to show that a representation $\rho:\Gamma\to G$ is primitive stable, it is sufficient to prove that $\tau_{\rho,x}$ takes all bi-infinite primitive geodesics of $\mathcal P_e$ to uniformly Morse quasigeodesics. 
We will say that $\mathcal P_e$ is \emph{$\sigma_{mod}$-regular} for $\rho$ if the subset of $G$ defined by 
\[ \mathcal P_e^{\rho}=\{ \rho(q(n)) \in G \ | \ q\in \mathcal P_e \text{ and } n\in \mathbb Z\} \]
is $\sigma_{mod}$-regular, i.e., any infinite sequence in $\mathcal P_e^\rho$ is $\sigma_{mod}$-regular. }
Then a necessary and sufficient condition for a representation to be $\sigma_{mod}$-primitive stable is as follows.

\begin{proposition}\label{iffcondition}
Let $\rho :\Gamma \rightarrow G$ be a representation of a non-abelian free group $\Gamma$ {\kim into a semisimple Lie group $G$.} Then $\rho$ is $\sigma_{mod}$-primitive stable if and only if the following holds:
\begin{itemize}
\item[(i)] ${\kim \mathcal P_e}$ is $\sigma_{mod}$-regular for $\rho$.
\item[(ii)] there exists a uniform constant $D>0$ such that for each $q:\mathbb Z\to G$ in ${\mathcal P_e}$, there exists a maximal flat $F_q$ such that $\rho(q(n))\cdot x$ remains at a bounded distance $D$ from the maximal flat $F_q$ for all $n\in\mathbb Z$.
\end{itemize}
\end{proposition}
\begin{proof}[Sketch of proof]
By the definition of primitive stable representation, (i) and (ii) immediately follow if $\rho$ is $\sigma_{mod}$-primitive stable.
Hence it is sufficient to prove the converse. Indeed its proof follows rather easily by the work of Kapovich, Leeb and Porti in \cite[Section 6.4.2]{KLP}.
Suppose that  (i) and (ii) hold for a representation $\rho :\Gamma \rightarrow G$.
{\kim Define a function $\Psi:\mathbb N \to \mathbb R$ by \[\Psi(n)=\min_{q\in \mathcal P_e} d(d_\Delta(x, \rho(q(n))\cdot x), \partial \Delta). \]
By (i), since any infinite sequence in $\mathcal P_e^\rho$ is $\sigma_{mod}$-regular, it easily follows that \[\lim_{n\to \infty}\Psi(n)=\infty.\]
}
In other words, for a given  constant $C>0$, there exists a uniform constant $R=R(\rho, C)$ such that
$$d(d_\Delta (x, \rho(w)\cdot x),\partial \Delta) \geq C$$ for any word $w$ such that {\kim $\rho(w) \in \mathcal P^\rho_e$ }with $|w|\geq R$. 


By a similar argument as in the proofs in \cite[Section 6.4.2]{KLP}, there exist uniform constants $L, A, D, S$ and $\Theta \subset \mathrm{int}(\sigma_{mod})$ such that $\tau_{\rho,x} \circ q$ is an $(L,A,\Theta,D,S)$-local Morse quasigeodesic for every bi-infinite primitive geodesic $q \in \mathcal P_e$. Then it follows from \cite[Theorem 7.18]{KLP} that $\tau_{\rho,x}\circ q$ is an $(L',A',\Theta',D')$-Morse quasigeodesic for some uniform constants $L', A', D'$ and $\Theta'\subset \mathrm{int}(\sigma_{mod})$ for every $q \in \mathcal P_e$. This completes the proof.
\end{proof}

Proposition \ref{thm:regular} implies that every positive representation satisfies (i) in Proposition \ref{iffcondition}. 
To verify that every positive representation is $\sigma_{mod}$-primitive stable, it only remains to check that (ii) in Proposition \ref{iffcondition} holds. 

\begin{theorem}\label{thm:positiveregular} Every positive representation of a compact, connected, orientable surface with one boundary component and negative Euler characteristic is $\sigma_{mod}$-primitive stable.
\end{theorem}
\begin{proof}
Let $\Sigma$ be a compact, connected, orientable surface with one boundary component and negative Euler characteristic. Let $\rho : \pi_1(\Sigma) \ra \pgl$ be a positive representation with a continuous $\rho$-equivariant positive map $\xi : \partial_\infty \pi_1(\Sigma) \ra \mathrm{Flag}(\sigma_{mod})$.
As mentioned above, it suffices to prove that condition (ii) in Proposition \ref{iffcondition} holds.

We claim that there exists a uniform constant $D=D(\rho,x)$ such that $$d_X(x,P(\xi(q^-),\xi(q^+))) \leq D$$ for all $q \in \mathcal P_e$. 
{\kim Suppose that the claim holds. Let $q:\mathbb Z\to \pi_1(\Sigma)$ be an arbitrary bi-infinite primitive geodesic in $\mathcal P_e$. Then for each $n\in \mathbb Z$, define a map $q[n] :\mathbb Z \to \pi_1(\Sigma)$ by \[ q[n](i)=q(n)^{-1}q(n+i).\] for $i\in \mathbb Z$. Obviously, $q[n]$ is a bi-infinite primitive geodesic and $q[n](0)=e$. Thus $q[n]\in \mathcal P_e$. Applying the claim to $q[n]$, we have \[d_X(x,P(\xi(q[n]^-),\xi(q[n]^+))) \leq D.\]
From the definition of $q[n]$, it directly follows that \[q[n]^-=q(n)\cdot q^- \text{ and } q[n]^+=q(n)\cdot q^+.\] Since $\xi$ is $\rho$-equivariant, we get
\begin{align*}
d_X(q(n)\cdot x, P(\xi(q^-),\xi(q^+)))  &= d_X(x,P(q(n)^{-1} \xi(q^-),q(n)^{-1}\xi(q^+))) \\
&= d_X(x,P(\xi(q[n]^-),\xi(q[n]^+)))   \leq D.
\end{align*}
This implies the condition (ii) in Proposition \ref{iffcondition}. Hence it is sufficient to prove the claim.


We now suppose that the claim dose not hold. Then there exists a sequence $(q_n : \mathbb Z\rightarrow \pi_1(\Sigma))$ of bi-infinite primitive geodesics in $\mathcal P_e$ such that  $$d_X( x, P(\xi(q_n^-),\xi(q_n^+))) \rightarrow \infty \text{ as } n\rightarrow \infty.$$
Applying Arzela-Ascoli theorem to $(q_n )$ which is a family of geodesics with $ q_n(0)=e$, by passing to a subsequence, we may assume that $q_n$ converges to a bi-infinite geodesic $q_\infty$ with $ q_\infty(0)=e$. Then $q_n^+$ and $q_n^-$ converge to $q_\infty^+$ and $q_\infty^-$ respectively.
Clearly $q_\infty^+$ and $q_\infty^-$ are distinct.
}

Let $\mathcal B_\infty(\Sigma)$ be the set of endpoints of the preimages of one boundary curve in the Cayley graph of $\pi_1(\Sigma)$.
If one of $q^\pm_\infty$ is in $\mathcal B_\infty(\Sigma)$, this means that a sequence of primitive elements corresponding to $q_n$ winds more and more around the one boundary component of $\Sigma$. Minsky \cite{Minsky} showed that this never happens due to the blocking property of the one boundary component of $\Sigma$. For this reason, it follows that $q^\pm_\infty \notin \mathcal B_\infty(\Sigma)$.

The $\rho$-equivariant continuous map $\xi :\partial_\infty \pi_1(\Sigma) \rightarrow \flagv$ is a one-to-one map on the complement of $\mathcal B_\infty(\Sigma)$. Since $q^\pm_\infty$ are distinct and $q^\pm_\infty \notin \mathcal B_\infty(\Sigma)$, the sequence of maximal flats $P(\xi(q_n^-),\xi(q_n^+))$ converges to a maximal flat $P(\xi(q_\infty^-),\xi(q_\infty^+))$ and thus $$d_X( x, P(\xi(q_n^-),\xi(q_n^+))) \rightarrow d_X( x, P(\xi(q_\infty^-),\xi(q_\infty^+)))  \text{ as } n \rightarrow \infty.$$ This contradicts the assumption that $d_X( x, P(\xi(q_n^-),\xi(q_n^+))) \rightarrow \infty \text{ as } n \rightarrow \infty.$ Therefore the claim holds.
\end{proof}

We prove Theorem \ref{mainthm} for positive representations and hence Theorem \ref{mainthm} now follows if we prove that the holonomies of convex projective structures are $\sigma_{mod}$-primitive stable.

\begin{proof}[Proof of Theorem \ref{mainthm}]
Let $\Sigma$ be a compact, connected, orientable surface with one boundary component and negative Euler characteristic. Let $\rho : \pi_1(\Sigma) \rightarrow \mathrm{PGL}(3,\mathbb R)$ be the holonomy of a convex projective structure on $\Sigma$. Then $\rho(\pi_1(\Sigma))$ is a discrete subgroup of $\rsl (3,\mathbb R)$ acting on a convex domain $\Omega$ in $\rp^2$ properly and freely. Furthermore it admits a $\rho$-equivariant continuous map $\xi : \partial_\infty \pi_1(\Sigma) \ra \flagv$. 
Note that $\rho(\gamma)$ is hyperbolic for any non-peripheral loop $\gamma \in \pi_1(\Sigma)$ and $\rho(c)$ is either hyperbolic or quasi-hyperbolic or parabolic for a peripheral loop $c$. When $\rho(c)$ is hyperbolic, $\xi$ is a $\rho$-equivariant, antipodal homeomorphism and thus $\rho$ is an Anosov representation. When $\rho(c)$ is quasi-hyperbolic, $\xi$ is a $\rho$-equivariant homeomorphism but not antipodal. When $\rho(c)$ is parabolic, as well-known, $\rho$ is a positive representation. We will show that $\rho$ is primitive stable in either case.

We apply Proposition \ref{iffcondition} to convex projective structures. First, (i) in Proposition \ref{iffcondition} immediately follows from Theorem \ref{cr}. Noting that every convex projective structure on $\Sigma$ admits an equivariant continuous map $\partial_\infty \pi_1(\Sigma) \ra \flagv$ which is a one-to-one map on $\partial_\infty\pi_1(\Sigma) \setminus \mathcal B_\infty(\Sigma)$, a proof for (ii) in Proposition  \ref{iffcondition} is exactly the same as the proof of Theorem \ref{thm:positiveregular}. Therefore by Proposition \ref{iffcondition}, every convex projective structure on $\Sigma$ is $\sigma_{mod}$-primitive stable.
\end{proof}



\begin{thebibliography}{99}
\bibitem{BGS} W. Ballmann, M. Gromov and V. Schroeder, Manifolds of nonpositive curvature, Progress in Mathematics, vol. 61, Birkh\"{a}user Boston Inc., Boston, MA, 1985.
\bibitem{Benoist}Y. Benoist, Propri\'et\'es asymptotiques des groupes lin\'eaires, Geom. Funct. Anal. 7 (1997) 1--47.
\bibitem{Be} Y. Benoist, Automorphismes des c{\^o}nes convexes, Invent. Math. 141 (2000), no. 1, 149--193.
\bibitem{BIHES}Y. Benoist, Convexes hyperboliques et fonctions quasisymétriques, Publications Math\'ematiques de
l’IH\'ES 97 (2003), no. 1, 181--237. 
\bibitem{Be2} Y. Benoist, Convexes divisibles I, TIFR. Stud. Math. 17 (2004) 339--374.
\bibitem{BH} Y. Benoist and D. Hulin, Cubic differentials and finite volume convex projective surfaces, Geom. Topol. 17 (2013), no. 1, 595--620.
\bibitem{BH2}Y. Benoist and D. Hulin, Cubic differentials and hyperbolic convex sets, J. Differential Geom. 98 (2014), no. 1, 1--19. 
\bibitem{Ben}Benz\'ecri, Sur les vari\'et\'es localement affines et projectives, Bull. Soc. Math. France 88 (1960), 229--332.
\bibitem{CG} J. P. Conze and Y. Guivarc'h, Limit sets of groups of linear transformations, Sankhy\=a Ser. A 62 (2000) no. 3, 367--385.
\bibitem{CLS} R. Canary, M. Lee and M. Stover, Amalgam Anosov representations, Geom. Topol. 21, no. 1 (2017), 215--251.
\bibitem{Crampon}M. Crampon, Dynamics and entropies of Hilbert metrics, Th\`ese, Universit\'e de Strasbourg, 2011.
\bibitem{Cr}M. Crampon, Lyapunov exponents in Hilbert geometry,  Ergodic Theory Dynam. Systems 34 (2014), 501--533.
\bibitem{Eb} P. Eberlein, Geometry of Nonpositively Curved Manifolds, Chicago Lectures in Mathematics, Chicago Univ. Press, Chicago (1996).
\bibitem{FG} V. Fock and A. Goncharov, Moduli spaces of local systems and higher Teichm\"uller theory, Publ. Math. Inst. Hautes \'{E}tudes Sci. no. 103 (2006), 1--211.
\bibitem{FG07} V. Fock and A. Goncharov,  Moduli spaces of convex projective structures on surfaces, Adv. Math. 208 (2007), no. 1, 249--273.
\bibitem{Fu1}P. Foulon, G\'eom\'etrie des \'equations diff\'erentielle du second ordre, Ann. Inst. Henri Poincar\'e, 45 (1986), 1--28.
\bibitem{Fu2}P. Foulon, Estimation de l'entropie des syst\`emes lagrangiens sans points conjugu\'es, Ann. Inst. Henri Poincar\'e Phys. Th\'eor., 57 (2) (1992), 117--146. With an appendix, ``About Finsler geometry''.
\bibitem{Gol} W. M. Goldman, Convex real projective structures on compact surfaces, J. Differential Geom. 31 (1990), no. 3, 791--845.
\bibitem{GGKW} F. Gu\'eritaud, O. Guichard, F. Kassel and A. Wienhard, Anosov representations and proper actions, Geom. Topol. 21 (2017), no. 2, 525--544.
\bibitem{GW} O. Guichard and A. Wienhard, Anosov representations : Domains of discontinuity and applications, Invent. Math.  190, no. 2, (2012), 357--438.
\bibitem{Hel} S. Helgason, Differential Geometry, Lie Groups, and Symmetric Spaces, American Mathematical Society, Providence, RI (1978).
\bibitem{HS}Y. Huang and Z. Sun, McShane identities for higher Teichm\"uller theory and the Goncharov-Shen potential, arXiv:1901.02032.
\bibitem{K}W. Jeon, I. Kim, K. Ohshika and C. Lecuire, Primitive stable representations of free Kleinian groups, Israel J. Math.  199, no. 2 (2014), 841--866.
\bibitem{KLM} M. Kapovich, B. Leeb and J. Millson, Convex functions on symmetric spaces, side lengths of polygons and the stability inequalities for weighted configurations at infinity, J. Differential Geom. 81 (2009), 297--354.
\bibitem{KLP}M. Kapovich, B. Leeb and J. Porti, Morse actions of discrete groups on symmetric spaces,  arXiv:1403.7671.
\bibitem{KLP2} M. Kapovich, B. Leeb and J. Porti, Anosov subgroups: dynamical and geometric characterizations. Eur. J. Math. 3 (2017), 808--898.
\bibitem{Kas} F. Kassel, Proper actions on corank-one reductive homogeneous spaces, J. Lie Theory 18 (2008), 961--978.
\bibitem{KL}I. Kim and M. Lee,  Separable-stable representations of compression body, Topology Appl. 206 (2016), 171--184.
\bibitem{La} F. Labourie, Anosov flows, surface groups and curves in projective space, Invent. Math. 165, no. 1, (2006), 51--114.
\bibitem{LM} F. Labourie and G. McShane, Cross ratios and identities for higher Teichm\"{u}ller-Thurston theory,  Duke Math. J. 149 (2009), no. 2, 279--345.
\bibitem{Link} G. Link, Geometry and dynamics of discrete isometry groups of higher rank symmetric spaces, Geom. Dedicata 122 (2006), 51--75. 
\bibitem{Lu}G. Lusztig, Total positivity in reductive groups, Lie theory and geometry, Progr. Math. 123, Birkh\"auser Boston, Boston, MA (1994), 531--568.
\bibitem{Marquis} L. Marquis,  Surface projective convexe de volume fini,  Ann. Inst. Fourier 62 (2012), 325--392.
\bibitem{Minsky} Y. Minsky, On dynamics of $\mathrm{Out}(\Gamma)$ on $\mathrm{PSL}_2(\mathbb C)$ characters, Israel J. of Math. 193 (2013), 47--70.
\bibitem{Wh1} J. H. C. Whitehead, On certain sets of elements in a free group, Proc. Lond. Math. Soc. 41 (1936), 48--56.
\bibitem{Wh2} J. H. C. Whitehead, On equivalent sets of elements in a free group, Annals of Math. 37 (1936), 782--800.
\end{thebibliography}
\end{document}